\documentclass[11pt, reqno]{amsart}

\usepackage{amsmath,amssymb,amsthm,amsfonts,mathtools,mathrsfs, mathscinet}
\usepackage[usenames,dvipsnames]{color}
\usepackage{url}
\usepackage{hyperref}
\hypersetup{colorlinks=true,citecolor=blue,linkcolor=Red,urlcolor=blue,pdfstartview=FitH}
\usepackage{ amscd, dsfont}
\usepackage[shortlabels]{enumitem}
\usepackage{easybmat,graphics}
\usepackage{etex, tikz, epic}
\usepackage{hyperref}
\hypersetup{colorlinks=true,citecolor=blue,linkcolor=black}
\usepackage{etex}
\usepackage[all]{xy}
\usepackage[mathscr]{euscript}

\newtheorem{theorem}{Theorem}[section]
\newtheorem{lemma}[theorem]{Lemma}
\newtheorem{corollary}[theorem]{Corollary}
\newtheorem{proposition}[theorem]{Proposition}

\newtheorem{remark}[theorem]{Remark}

\theoremstyle{definition}

\newlength{\margins}
\usepackage[top=\margins,bottom=\margins,left=\margins,right=\margins]{geometry}

\numberwithin{equation}{section}

\def\Ker{\operatorname{Ker}}

\def\Aut{\operatorname{Aut}}
\def\Out{\operatorname{Out}}
\def\Inn{\operatorname{Inn}}

\def\Hom{\operatorname{ Hom }}

\def\Char{\operatorname{char}}

\def\rank{\operatorname{rank}}

\def\Gal{\mathop{\rm Gal}\nolimits}
\def\SU{\mathop{\rm SU}\nolimits}

\def\GL{\mathop{\rm GL}\nolimits}
\def\Sym{\mathop{\rm Sym}\nolimits}
\def\SL{\mathop{\rm SL}\nolimits}
\def \Del{\mathop{\rm Del}\nolimits}
\def \Tr{\mathop{\rm Tr}\nolimits}
\def \Kaz{\mathop{\rm Kaz}\nolimits}

\def \Spec{\mathop{\rm Spec}\nolimits}

\newcommand{\Res}{\text{Res}}
\newcommand{\Fu}{{F_{un}}}

\newcommand{\FFu}{F'_{un}}
\newcommand{\FFsh}{\widehat{F'_{un}}}
\newcommand{\tDelta}{\tilde\Delta}

\newcommand{\Fsh}{{\widehat\Fu}}

\numberwithin{equation}{section}

\numberwithin{equation}{section}

\newcommand{\cA}{{\mathcal{A}}}
\newcommand{\cB}{{\mathcal{B}}}

\newcommand{\cF}{{\mathcal{F}}}

\newcommand{\cH}{{\mathcal{H}}}
\newcommand{\cI}{{\mathcal{I}}}

\newcommand{\cL}{{\mathcal{L}}}
\newcommand{\cM}{{\mathcal{M}}}

\newcommand{\cP}{\mathcal{P}}

\newcommand{\cT}{{\mathcal{T}}}
\newcommand{\cU}{{\mathcal{U}}}

\newcommand{\cX}{{\mathcal{X}}}
\newcommand{\cY}{{\mathcal{Y}}}

\newcommand{\bbG}{{\mathbb{G}}}

\newcommand{\bbQ}{{\mathbb{Q}}}
\newcommand{\bbR}{{\mathbb{R}}}

\newcommand{\bbZ}{{\mathbb{Z}}}

\newcommand{\tc}{\widetilde{c}}

\newcommand{\tp}{\widetilde{p}}

\newcommand{\ts}{\widetilde{s}}
\newcommand{\tu}{\widetilde{u}}
\newcommand{\tv}{\widetilde{v}}
\newcommand{\tw}{\widetilde{w}}
\newcommand{\tx}{\widetilde{x}}

\newcommand{\tB}{\widetilde{B}}

\newcommand{\tS}{\widetilde{S}}
\newcommand{\tT}{\widetilde{T}}
\newcommand{\tU}{\widetilde{U}}

\newcommand{\fI}{\mathfrak{I}}

\newcommand{\fO}{\mathfrak{O}}

\newcommand{\fQ}{\mathfrak{Q}}

\newcommand{\fm}{\mathfrak{m}}

\newcommand{\fp}{\mathfrak{p}}

\begin{document}
\title{Congruences of parahoric group schemes}
\author{Radhika Ganapathy}

\address{School of Mathematics, Tata Institute of Fundamental Research, Homi Bhabha Road, Colaba, Mumbai, India.}
\email{radhika@math.tifr.res.in}
\begin{abstract} Let $F$ be a non-archimedean local field and let $T$ be a torus over $F$. With $\cT^{NR}$ denoting the N\'eron-Raynaud model of $T$, a result of Chai and Yu asserts that the model  $\cT^{NR} \times_{\fO_F} \fO_F/\fp_F^m$ is canonically determined by $(\Tr_l(F), \Lambda)$ for $l>>m$, where $\Tr_l(F) = (\fO_F/\fp_F^l, \fp_F/\fp_F^{l+1}, \epsilon)$ with $\epsilon$ denoting the natural projection of $\fp_F/\fp_F^{l+1}$ on $\fp_F/\fp_F^l$, and $\Lambda:=X_*(T)$. In this article we  prove an analogous result for parahoric group schemes attached to facets in the Bruhat-Tits building of a connected reductive group over $F$.
\end{abstract}
\maketitle

\section{Introduction}

Let $F$ be a non-archimedean local field, $\fO_F$ its ring of integers, and $\fp_F$ its maximal ideal. Let $T$ be a torus over $F$. Such a torus is canonically determined by the lattice $\Lambda:=X_*(T)$ together with the action of $\Gamma_F = \Gal(F_s/F)$ on it (here $F_s$ is a separable closure of $F$).   
For large $m$, the action of $\Gamma_F$ on $\Lambda$ factors through the quotient $\Gamma_F/I_F^m$ of $\Gamma_F$, where $I_F^m$ is the $m$-th higher ramification subgroup (with upper numbering) of the inertia group $I_F$. This Galois group depends only on truncated data $\Tr_m(F):=(\fO_F/\fp_F^m, \fp_F/\fp_F^{m+1}, \epsilon)$, where $\epsilon$ is the natural projection of $\fp_F/\fp_F^{m+1}$ on $\fp_F/\fp_F^m$, via Deligne's theory; see (b) below.

Let $\cT^{NR}$ denote the N\'eron-Raynaud model of $T$ (see \cite{BLR}). The main result of \cite{CY01} asserts that $\cT^{NR} \times_{\fO_F} \fO_F/\fp_F^m$ is canonically determined by $(\Tr_l(F), \Lambda)$ for $l>>m$ (see Theorem 8.5 of \cite{CY01} for the precise statement; the parameters that $l$ depends on are also explicitly determined there). With $\cT$ denoting the neutral component of $\cT^{NR}$ this also implies that $\cT \times_{\fO_F} \fO_F/\fp_F^m$ is canonically determined by $(\Tr_l(F), \Lambda)$ with $l$ as above. From the point of view of Bruhat-Tits theory, when the connected reductive group is a torus, the model $\cT$ can be thought of as its Iwahori (or parahoric) group scheme. The purpose of this article is to prove an analogous result for parahoric group schemes attached to  facets in the Bruhat-Tits building of a connected reductive group over $F$. 

Our motivation for proving such a result arises naturally from the question of generalizing Kazhdan's theory of studying representation theory of split $p$-adic groups over close local fields to general connected reductive groups. Let us briefly recall the Deligne-Kazhdan correspondence:\\[4pt]
(a)  Given a local field $F'$ of characteristic $p$ and an integer $m \geq 1$, there exists a local field $F$ of characteristic 0 such that $F'$ is $m$-close to $F$, i.e., $\fO_F/\fp_F^m \cong \fO_{F'}/\fp_{F'}^m$.    \\[4pt]
(b) In \cite{Del84},  Deligne proved that if  $\Tr_m(F) \cong \Tr_m(F')$, then the Galois groups $\Gal(F_s/F)/I_F^m$ and $\Gal(F_s'/F')/I_{F'}^m$ are isomorphic.  This gives a bijection
\begin{align*}
& \text{\{Iso. classes of cont., complex, f.d. representations of $\Gal(F_s/F)$ trivial on $I_F^m$\}}\\
& \longleftrightarrow  \text{\{Iso. classes of cont., complex, f.d. representations of $\Gal(F_s'/F')$ trivial on $I_{F'}^m$\}}. 
\end{align*}
Moreover, all of the above holds when $\Gal(F_s/F)$ is replaced by $W_F$, the Weil group of $F$.\\[4pt]
(c)  Let $G$ be a split, connected reductive group defined over $\bbZ$. For an object $X$ associated to the field $F$, we will use the notation $X'$ to denote the corresponding object over $F'$. 
 In \cite{kaz86}, Kazhdan proved that  given $m \geq 1$, there exists $l \geq m$ such that if $F$ and $F'$ are $l$-close, then there is an algebra isomorphism $\Kaz_m:\mathscr{H}(G(F), K_m) \rightarrow \mathscr{H}(G(F'), K_m')$, where $K_m $ is the $m$-th usual congruence subgroup of $G(\fO_F)$. 
Hence, when the fields $F$ and $F'$ are sufficiently close, we have a bijection
 \begin{align*}
 &\text{\{Iso. classes of irr.  admissible representations $(\Pi, V)$ of $G(F)$ such that $\Pi^{K_m} \neq 0$\}} \\
  &\longleftrightarrow\text{\{Iso. classes of irr. admissible representations  $(\Pi', V')$ of $G(F')$ such that $\Pi'^{K_m'} \neq 0$\}}. 
\end{align*}
These results suggest that, if one understands the representation theory of $\Gal(F_s/F)$ for all  local fields $F$ of characteristic 0, then one can use it to understand the representation theory of $\Gal(F_s'/F')$ for a local field $F'$ of characteristic $p$, and similarly, with an understanding of the representation theory of ${G}(F)$ for all local fields $F$ of characteristic 0, one can study the representation theory of ${G}(F')$, for $F'$ of
characteristic $p$. This method has proved useful for studying the local Langlands correspondence for reductive $p$-adic groups in characteristic $p$ via the corresponding theory in characteristic 0 (see \cite{Bad02, Lem01,  Gan15, ABPS14, GV17}). An obvious observation, that goes into proving the Kazhdan isomorphism, is
\begin{align}\label{Kmiso0}
G(\fO_F)/K_m \cong G(\fO_F/\fp_F^m) \cong G(\fO_{F'}/\fp_{F'}^m)\cong G(\fO_{F'})/K_m' 
\end{align}
if the fields $F$ and $F'$ are $m$-close.

A useful variant of the Kazhdan isomorphism is now available for split reductive groups. Let $I$ be the standard Iwahori subgroup of $G$. It is shown in \cite{BT2} that there is a smooth affine group scheme $\cI$ defined over $\fO_F$ with generic fiber $G \times _\bbZ F$ such that $\cI(\fO_F)  =I.$ Define 
$I_m : = \Ker(\cI(\fO_F) \rightarrow \cI(\fO_F/\fp_F^m)).$
 In Section 3 of \cite{Gan15}, a presentation has been written down for this Hecke algebra $\cH(G, I_m)$ (extending Theorem 2.1 of \cite{How85} for $\GL_n$). Furthermore  if the fields $F$ and $F'$ are $m$-close, an argument of J.K.Yu (see Section 3.4.A of \cite{Gan15}) gives an isomorphism
\begin{align}\label{iiso1}
\beta: I/I_m &\rightarrow I'/I_m'.
\end{align}
Let us note here that unlike \eqref{Kmiso0}, the above isomorphism is not obvious since the group scheme $\cI$ is defined over $\fO_F$ and not over $\bbZ$. In fact the above isomorphism is obtained by proving that the reduction $\cI \times_{\fO_F} \fO_F/\fp_F^m$ depends only on $\Tr_m(F)$ and then evaluating it at the $\fO_F/\fp_F^m$-points. 
\ Using the presentation and this isomorphism, one gets an obvious map $\zeta_m : \cH(G(F), I_m) \rightarrow \cH(G(F'), I_m')$, when the fields $F$ and $F'$ are $m$-close (also see \cite{Lem01} for $\GL_n$), which was shown in \cite{Gan15} to be an isomorphism of rings. Hence we obtain a bijection 
\begin{align*}
 &\text{\{Iso. classes of irr. ad. representations $(\Pi, V)$ of $G(F)$ with $\Pi^{I_m} \neq 0$\}} \\
  & \longleftrightarrow\text{\{Iso. classes of irr. ad. representations  $(\Pi', V')$ of $G(F')$ with $\Pi'^{I_m'} \neq 0$\}}.
\end{align*}

When one wants to prove the Kazhdan isomorphism or its variant for general connected reductive groups, one is naturally led to consider parahoric subgroups, study the reduction of the underlying parahoric group schemes mod $\fp_F^m$, and prove that they are determined by truncated data. That is the goal of the present article.  Our proof is different from J.K.Yu's approach of proving \eqref{iiso1} for the Iwahori group scheme of a split $p$-adic group. We will use the construction of the parahoric group scheme via the Artin-Weil theorem (see \cite{La00}). Let us summarize the main results of this paper. 

First, given a split connected reductive group over $\bbZ$, one can unambiguously work with this group over an arbitrary field after base change. More generally, given a connected reductive group $G$ over $F$, we first need to make sense of what it means to give a group $G'$ over $F'$ where $F'$ is suitably close to $F$. Let us first explain how this is done for quasi-split groups. Let $(R,\Delta)$ be a based root datum and let $(G_0, T_0, B_0, \{u_\alpha\}_{\alpha \in \Delta})$ be a pinned, split, connected, reductive $\bbZ$-group with based root datum $(R, \Delta)$. We know that the $F$-isomorphism classes of quasi-split groups $G_q$ that are $F$-forms of $G_0$ are parametrized by the pointed cohomology set $H^1(\Gamma_F, Aut(R, \Delta))$ (see Theorem \ref{qsforms}). Let $E_{qs}(F, G_0)_m$  be the set of $F$-isomorphism classes of quasi-split groups $G_q$ that split (and become isomorphic to $G_0$) over an atmost $m$-ramified extension of $F$. It is easy to see that this is parametrized by the cohomology set
$H^1(\Gamma_F/I_F^m, Aut(R, \Delta))$. Using the Deligne isomorphism, we prove that there is a bijection $E_{qs}(F, G_0)_m \rightarrow E_{qs}(F', G_0')_m,\;\, G_q \rightarrow G_q',$ provided $F$ and $F'$ are $m$-close.  Moreover, with the cocycles chosen compatibly, this will yield data $(G_q, T_q, B_q)$ over $F$ (where $T_q$ is a maximal $F$-torus and $B_q$ is an $F$-Borel containing $T_q$), and correspondingly $(G_q', T_q', B_q')$ over $F'$, together with an isomorphism $X_*(T_q) \rightarrow X_*(T_q')$ that is $\Del_m$-equivariant (see Lemma \ref{charQSCLF}). It is a simple observation that the maximal $F$-split subtorus $S_q$ of $T_q$ is a maximal $F$-split torus in $G_q$ (see Lemma \ref{maxFsplit}). We prove that there is a simplicial isomorphism between the apartments $\cA_m: \cA(S_q,F) \rightarrow \cA(S_q', F')$ if the fields $F$ and $F'$ are $m$-close (see Proposition \ref{ACLF} for precise statement). Let $\cF$ be a facet in $\cA(S_q,F)$ and $\cF' = \cA_m(\cF)$. Then $\cF'$ is a facet in $\cA(S_q',F')$. We prove that the parahoric group schemes $\cP_{\cF} \times_{\fO_F} \fO_F/\fp_F^m$ and $\cP_{\cF'} \times_{\fO_{F'}} \fO_{F'}/\fp_{F'}^m$ are isomorphic provided $F$ and $F'$ are $l$-close for $l>>m$ (see Theorem \ref{QSP} and Proposition \ref{QSDP} for precise statements).  To prove this theorem, we prove an analogous statement for the root subgroup schemes if the fields $F$ and $F'$ are sufficiently close, invoke the result of Chai-Yu (see \cite{CY01}) that the reduction of the (lft) N\'eron models of the corresponding tori are isomorphic if the fields are sufficiently close, and use the Artin-Weil theorem on obtaining group schemes as solutions to birational group laws.

To move to the general case, we recall that any connected reductive group is an inner form of a quasi-split group, and the $F$-isomorphism classes of inner forms of $G_q$ is parametrized by the cohomology set $H^1(\Gal(\Fu/F), G_q^{ad}(\Fu))$ (where $\Fu$ is the maximal unramified extension of $F$ contained in $F_s$).  With $G_q'$ corresponding to $G_q$ as above, we prove in Lemma \ref{IF} that
\[H^1(\Gal(\Fu/F), G_q^{ad}(\Fu)) \cong H^1(\Gal(\FFu/F'), G_q^{ad}(\FFu))\]
as pointed sets if the fields $F$ and $F'$ are $m$-close using the work of Kottwitz (see \cite{Kot16}). Using the work of Debacker-Reeder \cite{DR} it is further possible to refine the above and obtain an isomorphism at the level of cocycles (see Section \ref{IFCCLF}). All the above yields data $(G, S, A)$ where $G$ is a connected reductive group over $F$ that is an inner form of $G_q$, a maximal $\Fu$-split $F$-torus $S$ that contains a maximal $F$-split torus $A$ of $G$, and similarly $(G',S',A')$ over $F'$, together with a $\Gal(\Fsh/F)$-equivariant simplicial isomorphism $\cA_{m,*}: \cA(S,\Fsh) \rightarrow \cA(S',\FFsh)$  (see Corollary \ref{IFACLF}). Here $\Fsh$ denotes the completion of $\Fu$. Let $\tilde\cF_*$ be a $\Gal(\Fsh/F)$-invariant facet in $\cA(S,\Fsh)$ and let $\tilde\cF'_* = \cA_{m,*}(\tilde\cF_*)$. We prove that there is a $\Gal(\Fsh/F)$-equivariant isomorphism 
\[\tilde p_{m,*}: \cP_{\tilde\cF_*} \times_{\fO_\Fsh} \fO_\Fsh/\fp_\Fsh^m \rightarrow \cP_{\tilde\cF'_*} \times_{\fO_{\FFsh}} \fO_{\FFsh}/\fp_{\FFsh}^m\]
 provided $F$ and $F'$ are $l$-close (see Proposition \ref{EDPCLF}). With $\cF_* :=(\tilde\cF_*)^{\Gal(\Fsh/F)}$ and $\cF'_* :=(\tilde\cF'_*)^{\Gal(\FFsh/F')}$, the above descends to an isomorphism of group schemes 
\[p_{m,*}: \cP_{\cF_*} \times_{\fO_F} \fO_F/\fp_F^m \rightarrow \cP_{\cF'_*} \times_{\fO_{F'}} \fO_{F'}/\fp_{F'}^m.\]
As a corollary, we obtain that
\[\cP_{\cF_*} (\fO_F/\fp_F^m) \cong \cP_{\cF'_*} (\fO_{F'}/\fp_{F'}^m)\] as groups provided the fields $F$ and $F'$ are $l$-close.

\section*{Acknowledgments}
I would like to express my gratitude to J. K. Yu for introducing me to questions related to this article and for the insightful discussions during my graduate school years. I thank the referees for the helpful comments and suggestions, particularly regarding removing the assumption on residue characteristic that was made in an earlier version of this article. I thank Dipendra Prasad and K. V.  Shuddhodan for the interesting discussions and for their encouragement. 

\section{Some review}
Unless otherwise stated, $F$ will denote a non-archimedean local field, that is, a complete discretely valued field with perfect residue field. Let $\fO_F$ denote its ring of integers, $\fp_F$ its maximal ideal, $\omega=\omega_F$ an additive valuation on $F$ normalized so that $\omega(F) = \bbZ$, and $\pi=\pi_F$ a uniformizer. Fix a separable closure $F_s$ of $F$ and let $\Gamma_F = \Gal(F_s/F)$.
\subsection{Deligne's theory}\label{Deligne}
Let $m \geq 1$.  Let $I_F$ be the inertia group of $F$ and  $I_F^m$ be its $m$-th  higher ramification subgroup with upper numbering (cf. Chapter IV of \cite{Ser79}). Let us summarize the results of Deligne \cite{Del84} that will be used later in this article.  Deligne considered the triplet $\Tr_m(F) = (\fO_F/\fp_F^m, \fp_F/\fp_F^{m+1}, \epsilon)$, where $\epsilon$ is the natural projection of $\fp_F/\fp_F^{m+1}$ on $\fp_F/\fp_F^m$, and proved that $\Gamma_F/I_F^m$ is canonically determined by $\Tr_m(F)$. Hence an isomorphism of triplets $\psi_m: \Tr_m(F) \rightarrow \Tr_m(F')$ gives rise to an isomorphism
\begin{equation}\label{Deliso}
\Gamma_F/I_F^m \xrightarrow{\Del_m} \Gamma_{F'}/I_{F'}^m
\end{equation}
that is unique up to inner automorphisms (see Equation 3.5.1 of \cite{Del84}). More precisely, given an integer $f \geq 0$, let $ext(F)^f$ denote the category of finite separable extensions $E/F$  satisfying the following condition: The normal closure $E_1$ of $E$ in $F_s$ satisfies $\Gal(E_1/F)^f = 1$. Deligne proved that an isomorphism $\psi_m: \Tr_m(F) \rightarrow \Tr_m(F')$ induces an equivalence of categories
$ext(F)^m \rightarrow ext(F')^m$. Here is a partial description of the map $\Del_m$ (see Section 1.3 of \cite{Del84}).  Let $L$ be a finite totally ramified Galois extension of $F$ satisfying $I(L/F)^m = 1$ (here $I(L/F)$ is the inertia group of $L/F$). Then  $L = F(\alpha)$ where $\alpha$ is a root of an Eisenstein polynomial \[P(x) = x^n + \pi \sum a_ix^i\] for $a_i \in \fO_F$.  Let $a_i' \in \fO_{F'}$ be such that $a_i \mod \fp^m \rightarrow a_i' \mod \fp'^m$. So $a_i'$ is well-defined mod $\fp'^m$.
Then the corresponding extension $L'/F'$ can be obtained as $L' = F'(\alpha')$ where $\alpha'$ is a root of the polynomial \[P'(x) = x^n + \pi' \sum a_i'x^i\] where $\pi \mod \fp_F^m \rightarrow \pi' \mod \fp_{F'}^m$.  The assumption that $I(L/F)^m =1$ ensures that the extension $L'$ does not depend on the choice of $a_i'$, up to a unique isomorphism.

\subsection{The main theorem of Chai-Yu} \label{CY}
Let $T$ be a torus over $F$, and let $K/F$ be a Galois extension such that $T$ is split over $K$. Let $\Gamma_{K/F} = \Gal(K/F)$ and let $\Lambda = X_*(T)$, the co-character group of $T$. Then $T$ is determined by the $\Gamma$-module $\Lambda$ upto a canonical isomorphism. Let $F'$ denote another non-archimedean local field, and we will denote the analogous objects over $F'$ with a superscript $'$. We introduce the following series of congruence notation.
\begin{itemize}[leftmargin=*]
\item $(\fO_F, \fO_K) \equiv_{\psi_m} (\fO_{F'}, \fO_{K'}) \;(level\;\; m):$\\
 This means that $\psi_m$ is an isomorphism $\fO_K/\pi^m\fO_K\rightarrow  \fO_{K'}/\pi'^m\fO_{K'}$ and induces an isomorphism $\fO_F/\pi^m\fO_{F} \rightarrow \fO_{F'}/\pi'^m\fO_{F'}$. We denote this induced isomorphism also by $\psi_m$. Having chosen the uniformizers, this also induces an isomorphism $\Tr_m(F) \rightarrow \Tr_m(F')$, which we still denote by $\psi_m$.
\item $(\fO_F, \fO_{K}, \Gamma_{K/F}) \equiv_{\psi_m,\gamma} (\fO_{F'}, \fO_{K'}, \Gamma_{K'/F'})\; (level\;\; m)$:\\ This means $(\fO_F, \fO_K) \equiv_{\psi_m} (\fO_{F'}, \fO_{K'}) (level\;\; m)$, $\gamma$ is an isomorphism $\Gamma_{K/F} \rightarrow \Gamma_{K'/F'}$, and $\psi_m$ is $\Gamma_{K/F}$-equivariant relative to $\gamma$. 
\item $(\fO_F, \fO_K, \Gamma_{K/F},\Lambda) \equiv_{\psi_m,\gamma, \lambda} (\fO_{F'}, \fO_{K'}, \Gamma_{K'/F'},\Lambda') \;(level\;\; m)$:\\ This means $(\fO_F, \fO_K, \Gamma_{K/F}) \equiv_{\alpha,\beta} (\fO_{F'}, \fO_{K'}, \Gamma_{K'/F'}) \;(level\;\; m)$ and $\lambda$ is an isomorphism $\Lambda \rightarrow \Lambda'$ which is $\Gamma_{K/F}$-equivariant relative to $\gamma$.
\end{itemize}
We say that ``$X$ is determined by $(\fO_F/\pi^m\fO_F, \fO_K/\pi^m\fO_K, \Gamma_{K/F}, \Lambda)''$ to mean that if
\[(\fO_F, \fO_K, \Gamma_{K/F},\Lambda) \equiv_{\psi_m,\gamma, \lambda} (\fO_{F'}, \fO_{K'}, \Gamma_{K'/F'},\Lambda')\; (level\;\; m)\]
then there is a canonical $\Gamma_{K/F}$-equivariant isomorphism $X \rightarrow X'$ determined by $(\psi_m, \gamma, \lambda)$. 

Let $\cT^{NR}$ denote the N\'eron-Raynaud model of $T$ considered in \cite{CY01}. This is a smooth model of $T$ with connected generic fiber such that $\cT^{NR}(\fO_\Fsh)$ is the maximal bounded subgroup of $T(\Fsh)$, where $\Fsh$ is the completion of the maximal unramified extension $\Fu$ of $F$ contained in $F_s$. This model is of finite type over $\fO_F$. 
\begin{theorem}[Theorem 8.5 of \cite{CY01}]  Let $m \geq 1$. There exists $l \geq m$ such that the model 
\[\cT^{NR} \times_{\fO_F} \fO_F/\fp_F^m\text{ is determined by } (\fO_F/\pi^l\fO_F, \fO_K/\pi^l\fO_K, \Gamma_{K/F}, \Lambda).\]
\end{theorem}
The parameters that $l$ depends on are also explicitly determined in Theorem 8.5 of \cite{CY01}. Let $\cT$ denote the neutral component of $\cT^{NR}$. This is a smooth model over $\fO_F$ with connected generic and special fibers, and is of finite type over $\fO_F$. Its $\fO_{\Fsh}$-points is the Iwahori subgroup of $T(\Fsh)$. 
\begin{lemma}\label{torusCLF}  Let $\cT$,  $l \geq m$ as above. Then the model 
\[\cT \times_{\fO_F} \fO_F/\fp_F^m\text{ is determined by } (\fO_F/\pi^l\fO_F, \fO_K/\pi^l\fO_K, \Gamma_{K/F}, \Lambda).\]
\end{lemma}
\begin{proof}

This lemma follows from Lemma 8.5 of \cite{CY01} and the observation that the formation of $\cT$ commutes with any base change on $\Spec(\fO_F)$, that is, 
\[ (\cT^{NR} \times_{\fO_F} \fO_F/\fp_F^m)^0 = \cT \times_{\fO_F} \fO_F/\fp_F^m.\qedhere \]
\end{proof}

When the connected reductive group is a torus $T$, the  model $\cT$ is its Iwahori (or parahoric) group scheme. We will study congruences of parahoric group schemes attached to facets in the Bruhat-Tits building of a connected reductive group $G$ over $F$. To this end, let us recall some results from Bruhat-Tits theory and the construction of parahoric group schemes (using Artin-Weil theorem, following \cite{La00}), that will used later in this article.

 Given a  connected reductive group $G$ over $F$, let $G^{der}$ denote the derived subgroup of $G$, and  $G^{ad}$ its adjoint group. Let $\cB(G,F)$  denote the reduced Bruhat-Tits building of $G$ over $F$, that is, the building of $G^{der}$ over $F$. The building is obtained by gluing together apartments $\cA(S,F)$ where $S$ runs over the maximal $F$-split tori in $G$. The apartment $\cA(S,F)$ is an affine space under $X_*(S^{der}) \otimes_\bbZ \bbR$ where $S^{der} = S \cap G^{der}$. Let $\cF$ be a facet in $\cB(G,F)$ and let $P_\cF$ denote the parahoric subgroup of $G(F)$ attached to $\cF$. Bruhat-Tits show that there exists a smooth affine $\fO_F$-group scheme $\cP_\cF$ with generic fiber $G$ such that $\cP_\cF(\fO_F) = P_\cF$. We recall the construction of $\cP_\cF$, following Landvogt (\cite{La00}). The parahoric group scheme is first constructed over $\Fsh$ (note that $G_{\Fsh}$ is quasi-split), and the model over $F$ is obtained using \'etale descent.
 
\subsection{Structure of quasi-split groups}
Let $G$ denote a quasi-split connected reductive group over $F$.  Let $S$ be a maximal $F$-split torus in $G$ and let $T$ (resp. $N$) be the centralizer (resp. normalizer) of $S$ in $G$. Let $B$ be an $F$-Borel subgroup of $G$ with $T \subset B$. Note that $T$ is a maximal $F$-torus in $G$.   Further $G$ and $T$ split over $F_s$ and the Galois group $\Gamma_F$ acts on the group of characters $X^*(T)$ of $T$, preserves the root system $\Phi(G,T)$ of $T$ in $G$, and also the base $\tilde{\Delta}$ of $\Phi(G,T)$ associated to the Borel subgroup $B$. Let $K \subset F_s$ denote the smallest sub-extension of $F_s$ splitting $T$ (and hence $G$). Let $\Phi(G,S)$ denote the set of roots of $S$ in $G$.

\subsubsection{Root subgroups $U_a, a \in \Phi(G, S)$}\label{CSsystem} The elements of $\Phi(G,S)$ are restrictions of elements of $\Phi(G,T)$ to $S$, and the restrictions to $S$ of the elements of $\tilde{\Delta}$ form a basis $\Delta$ of $\Phi(G,S)$. Moreover, the elements of $\tilde{\Delta}$ that have the same restriction to $S$ form a single Galois orbit for the action of $\Gamma_F$ on $\tilde{\Delta}$. 
For $\alpha \in \Phi(G,T)$, let $\tU_\alpha$ be the corresponding root subgroup of $G_{K}$. The group $\Gamma_{K/F}$ permutes $\tU_\alpha$ and $\gamma(\tU_\alpha) =  \tU_{\gamma(\alpha)}$. Let $\Sigma_\alpha$ be the stabilizer of $\tU_\alpha$ and let $L_\alpha$ be the corresponding field of invariants. We say that $L_\alpha$ is the \textit{field of definition} of $\alpha$. Note that $\tU_\alpha$ is defined over $L_\alpha$ by Galois descent. Let $\{\tx_\alpha:\bbG_{a,L_\alpha} \rightarrow \tU_\alpha\;|\; \alpha \in \Phi(G,T)\}$ denote a Chevalley-Steinberg splitting of $G$. It has the following properties.

\begin{enumerate}[leftmargin=*]
\item If the restriction $a$ of $\alpha \in \Phi(G,T)$ to $S$ is an indivisible element of $\Phi(G,S)$, then $\tx_\alpha$ is an $L_\alpha$-isomorphism of $\bbG_a$ to $\tU_\alpha$ and we have $\tx_{\gamma(\alpha)} = \gamma \circ \tx_\alpha \circ \gamma^{-1}$ for each $\gamma \in \Gal(K/F)$. 
\item If the restriction $a$ of $\alpha \in \Phi(G,T)$ to $S$ is divisible, then there exists two distinct roots $\beta, \beta' \in \Phi(G,T)$ of restriction $a/2$ to $S$ such that $\alpha = \beta +\beta'$; we have $L_\beta = L_{\beta'}$, $L_\beta$ is a quadratic separable extension of $L_\alpha$ and for each $\gamma \in \Gal(K/F)$ there exists $\epsilon = \pm 1$ such that $\gamma \circ \tx_\alpha(u) \circ \gamma^{-1} = \tx_{\gamma(\alpha)}(\epsilon u)$; if $\gamma \in \Gal(K/L_{\alpha})$, we have $\epsilon = -1$ if and only if $\gamma$ induces the unique non-trivial automorphism of $L_\beta$. 
\end{enumerate} 
 Now we describe all possible structures for the root subgroups $U_a, a \in \Phi(G,S)$.  We may and do assume that $a \in \Delta$. Let $\tilde{\Delta}_a$ be the orbit of $\Gamma_{K/F}$ in $\tilde{\Delta}$. Let $\pi: G^a \rightarrow \langle U_a, U_{-a}\rangle$ be the universal cover of the semisimple group generated by $U_a$ and $U_{-a}$. The classification of Dynkin diagrams gives two possible cases:
\begin{enumerate}[wide, labelwidth=!, labelindent=0pt]
\item[\textbf{Case} I.] The group $G_K^a$ is isomorphic to a product of the groups $\SL_2$ indexed by $\tilde{\Delta}_a$ and are permuted transitively by $\Gal(K/F)$, the field of definition of the factor of index $\alpha$ is $L_\alpha$ and $G^a \cong \mathrm{Res}_{L_\alpha/F} \SL_2$. Then $U_a \cong \mathrm{Res}_{L_{\alpha}/F} \tU_\alpha$ for $\alpha \in \tilde{\Delta}_a$. If $\tx_\alpha:L_\alpha \rightarrow \tU_\alpha$, then $x_a = \Res_{L_{\alpha}/F} \tx_\alpha$ is a $F$-isomorphism of $\Res_{L_{\alpha}/F} \bbG_a$ to $U_a$; the pair $(L_\alpha, x_a)$ is called a pinning of $U_a$.  Via $x_a$, we obtain an isomorphism of $L_\alpha$ with $U_a(F)$, which we also denote by $x_a$. If $(\tx_\beta)_{\beta \in \tilde{\Delta}}$ is an Chevalley-Steinberg splitting of $G$, then we have for each $u \in L_\alpha$,
\begin{align}\label{2anotroot}
x_a(u) = \prod_{\beta \in \tilde{\Delta}_a} \tx_\beta(u_\beta)
\end{align}
In the above, $\beta = \gamma(\alpha)$ for some $\gamma \in \Gamma_{K/F}$ and $u_\beta := \gamma(u)$. 
The subgroups $U_{-a}$ and the splitting $x_{-a}$ are obtained using $U_{-\alpha}$ and $\tx_{-\alpha}$ analogously. 
\item[\textbf{Case} II.] The group  $G_K^a$  is isomorphic to a product of the groups $\SL_3$ indexed by the set $I$ consisting of pairs of two elements $\{\alpha, \bar\alpha\}$ of $\tilde{\Delta}_a$ such that $\alpha+\bar\alpha$ is a root. We have $L_\alpha = L_{\bar\alpha}$, $L_\alpha$ is a quadratic extension of $L_{\alpha+\bar\alpha}$. The simple factor $\bar G$ of index $\{\alpha,\bar\alpha\}$ is defined over $ L_{\alpha+\bar\alpha}$, split over $ L_{\alpha}$, and is isomorphic over $L_{\alpha+\bar\alpha}$ to the special unitary group of the Hermitian form $h:(x_{-1}, x_0, x_1) \rightarrow \tau(x_{-1})x_1+\tau(x_0)x_0 + \tau(x_1)x_{-1}$ over $L^3$. Here $\tau$ is the unique non-trivial element of $\Gal(L_\alpha/L_{\alpha+\bar\alpha})$. We denote this simple factor as $\SU_3$, and then $G^a \cong \mathrm{Res}_{L_{\alpha+\bar\alpha}/F} \SU_3$.   

Let $H_0(L_\alpha,L_{\alpha+\bar\alpha}) := \{(u,v) \in L_\alpha \times L_\alpha\;|\; v+\tau(v) = u\tau(u)\}$ denote the $L_{\alpha+\bar\alpha}$-group with group law $(u,v) \cdot (\tu,\tv) = (u+\tu, v+\tv+ \tau(u)\tu).$ Then $\zeta : (u,v)\rightarrow \tx_\alpha(u) \tx_{\alpha+\bar\alpha}(-v)\tx_{\bar\alpha}(\tau(u))$
is an $L_{\alpha+\bar\alpha}$-group isomorphism of $H_0(L_\alpha,L_{\alpha+\bar\alpha})$ with the subgroup $\bar U = \tU_\alpha \tU_{\alpha+\bar\alpha}\tU_{\bar\alpha}$ of $\bar G$. Then $U_a = \Res_{L_{\alpha+\bar\alpha}/F}\bar U$ and $x_a =  \Res_{L_{\alpha+\bar\alpha}/K}\zeta$ is an $F$-isomorphism of groups $H(L_\alpha,L_{\alpha+\bar\alpha}) = \Res_{L_{\alpha+\bar\alpha}/F} H_0(L_\alpha,L_{\alpha+\bar\alpha})$ with $U_a$. Further, for $(u,v) \in L_\alpha \times L_\alpha$,
\[x_a(u,v) = \prod \tx_{\beta}(u_\beta) \tx_{\beta+\bar\beta}(-v_\beta) \tx_{\bar\beta}(\tau(u_\beta))\]
In the above, for each $\beta$, we choose $\gamma \in \Gal(K/F)$ such that $\beta = \gamma(\alpha)$; then $\bar\beta = \gamma(\bar\alpha)$, $\tx_\beta = \gamma \circ \tx_{\alpha} \circ \gamma^{-1},\; \tx_{\bar\beta} = \gamma \circ \tx_{\bar\alpha} \circ \gamma^{-1},\; \tx_{\beta+\bar\beta} = \gamma \circ \tx_{\alpha+\bar\alpha} \circ \gamma^{-1},\; u_\beta = \gamma(u), v_\beta = \gamma(v)$.  

Note that the root subgroup $U_{2a}(K)$ associated to the root $2a$ consists of elements $x_a(0,v)$ where $v \in L_\alpha^0:=\{v \in L_\alpha\;|\; v +\tau(v)=0\}$ and the map $v \rightarrow x_a(0,v)$ is an $F$-vector space isomorphism of $L^0$ with $U_{2a}(K)$. 
\end{enumerate}

\subsubsection{On the splitting extension of the root}\label{splitting extension} Let $a \in \Phi^{red}(G,S)$ with $2a$ is not a root. We fix a pinning $(L_\alpha, x_a)$ of $U_a$ where $\alpha \in \tDelta_a$ as in (I) above. The subset of endomorphisms of the $F$-vector space $U_a$ of the form $\mu_{x_a}(t): x_a(u)\rightarrow x_a(tu)$ for $t \in L_\alpha$ does not depend on the choice of $(L_\alpha,x_a)$ (see Section 4.1.8 of \cite{BT2}). This is denoted by $L_a$ and is called the field attached of the root $a$.  It is isomorphic to $L_\alpha$ via the map $t \rightarrow \mu_{x_a}(t)$.  Its inverse gives an embedding of $L_a \hookrightarrow K$. A similar definition is obtained when $2a$ is a root in Section 4.1.14 of \cite{BT2}. 

\subsubsection{Valuations}
Let $\omega:F \rightarrow \bbR^\times$ be as in Section \ref{Deligne}, and we denote its extension to $K$  also as $\omega$.
The notion of valuation of root datum was defined in \cite{BT1}. For $\alpha \in \Phi(G,T)$, and put
\[ \phi_\alpha(\tx_\alpha(u)) = \omega(u), u \in K^\times.\] Then $\tilde{\phi} = (\phi_\alpha)_{\alpha \in \Phi(G,T)}$ defines a valuation of the root datum $(T_K, (\tU_\alpha)_{\alpha \in \Phi(G,T)})$ in the group $G(K)$ (recall that $G_K$ is split). It is shown in \cite{BT2} that $\tilde{\phi}$ descends to $(T, (\tU_a)_{a \in \Phi(G,S)})$ and defines a valuation on it. We explicitly define $\phi_a: U_a(F) \backslash \{1\} \rightarrow \bbR$ from $\tilde{\phi}$. For $a \in \Phi(G,S)$, let $A$ (resp. $B$) be the set of $\alpha \in \Phi(G,T)$ whose restriction to $S$ is $a$ (resp. $2a$). For $u \in U_a(F)$, there exist unique $\tilde u_\alpha$ such that $u = \prod_{\alpha \in A \cup B} \tilde u_\alpha$
for an arbitrary ordering of $A \cup B$ and we put
\[ \phi_a(u) = \inf\left(\inf_{\alpha \in A} \tilde{\phi}_\alpha(\tilde u_\alpha), \inf_{\alpha \in B} \frac{1}{2}\tilde{\phi}_\alpha(\tilde u_\alpha)\right).\]
This number is independent of the choice of ordering of $A \cup B$. 
Then $\phi = (\phi_a)_{a \in \Phi(G,S)}$ defines a valuation of root datum on $(T, (U_a)_{a \in \Phi(G,S)})$ (see Section 4.2.2 of \cite{BT2}).

\subsection{Parahoric group schemes; quasi-split descent}

In this section, we assume that $F$ is also strictly Henselian, that is its residue field is separably closed. 
\subsubsection{Affine root system and the associated Weyl groups}\label{ARS} The apartment $\cA(S,F)$ can also be thought of as the set of valuations that are equipollent to $\phi =(\phi_a)_{a \in \Phi(G,S)}$, where $\phi$ as above. This is an affine space under $X_*(S^{der}) \otimes_\bbZ \bbR$ and $N(F)$ acts on it by affine transformations (see Section 6.2.2 of \cite{BT1}). Let us denote the point of $\cA(S,F)$ corresponding to $\phi$ as $x_0$. For $a \in \Phi(G,S)$, let $\Gamma_a = \phi_a(U_a(F) \backslash \{1\})$ and
\[\tilde\Gamma_a = \{\phi_a(u)\; |\; u \in U_a(F) \backslash \{1\},\; \phi_a(u) = \sup \phi_a(uU_{2a}(F))\}.\]
Here we have used the convention that $U_{2a} =1$ if $2a$ is not a root. Let
\[\Phi^{af}(G,S) = \{ \psi: \cA(S,F) \rightarrow \bbR\;|\;\psi(\cdot ) = a(\cdot -x_0) +l, a \in \Phi(G,S),\; l \in \tilde\Gamma_a\}\]
denote the set of affine roots of $S$ in $G$. 
Choosing $x_0$ allows us to identify $A(S,F)$ with $X_*(S^{der}) \otimes_\bbZ \bbR$. With this identification, the vanishing hyperplanes coming from $\Phi(G,S)^{af}$ makes $\cA(S,F)$ into a (poly)simplicial complex.  The group generated by reflections through the hyperplanes coming from $\Phi(G,S)^{af}$ is the affine Weyl group denoted by $W^{af}$.  The extended affine Weyl group is defined as $W^e: = N(F)/T(F)_1$ where 
$T(F)_1$ is the kernel of the Kottwitz homomorphism $\kappa_T: T(F) \rightarrow X^*(\hat T^{I_F}) = X_*(T)_{I_F}$ (see \cite{HR08}).
With $W: = W(G,S)$, the group $W^e$ hence fits into an exact sequence
\[1 \rightarrow X_*(T)_{I_F} \rightarrow W^e \rightarrow W \rightarrow 1.\]
\subsubsection{The associated root subgroup schemes} \label{RSGS}Let us recall the filtrations on root subgroups and the associated root subgroup schemes from Section 4.3 of \cite{BT2}. For $a \in \Phi(G,S)$, let $\phi_a: U_a(F) \rightarrow \bbR \cup \{\infty\}$ be as above. For $k \in \bbR$, Let $U_{a,k} = \{u \in U_a(F)\;|\; \phi_a(u) \geq k\}$. Next, let us describe the associated root subgroup schemes. 
\begin{enumerate}[wide, labelwidth=!, labelindent=0pt]
\item[\textbf{Case} I.] Let $a \in \Phi^{red}(G,S)$ such that $2a \notin \Phi(G,S)$. For $k \in \tilde\Gamma_a$, let $L_{a,k} = \{ u \in L_a\;|\; \omega(u) \geq k\}$. Then $L_{a,k}$ is a free $\fO_{F}$-module of finite type. Let $\cL_{a,k}$ be the canonical smooth $\fO_{F}$-group scheme associated to this module (More precisely, given a free $\fO_F$-module $M$ of finite type, the functor taking any $\fO_F$-algebra $R$ to the additive group $R \otimes M$ is representable by a smooth $\fO_F$-group scheme $\cM$ whose affine algebra is identified with the symmetric algebra of the dual of $M$). Let $U_{a,k}$ be the image under $x_a$ of $L_{a,k}$ and let $\cU_{a,k}$  be the $\fO_F$-group scheme obtained by transport of structure using $x_a$. Then $\cU_{a,k}$ has generic fiber $U_a$ and $\cU_{a,k}(\fO_F) = U_{a,k}$. The definition is extended to $k \in \bbR\backslash \{0\}$ in Section 4.3.2 of \cite{BT2}.
\item[\textbf{Case} II.] Let $a \in \Phi^{red}(G,S)$ with $2a \in \Phi(G,S)$. The root subgroup $U_a \cong \Res^{L_{2a}}_F H_0(L_a,L_{2a})$ via $x_a$. In order to describe the root subgroup schemes of the filtration $U_{a,k}$, we use an alternate description of $H_0(L_a,L_{2a})$.  Recall that $L_a^0$ is the set of trace 0 elements of $L_a$. Let $L_a^1$ denote the set of trace 1 elements in $L_a$ and let 
\[(L_a)_{max}^1:=\{ \lambda \in L_a^1\;|\; \omega(\lambda) = \sup\{\omega(x)\;|\; x \in L_a^1\}\}.\]
Note that $(L_a)_{max}^1 \neq \emptyset$ and when the residue field of $L_a$ is of characteristic $\neq 2$, $1/2 \in (L_a)_{max}^1$. Let $\lambda \in (L_a)_{max}^1$ and let  $H_0^\lambda:=L_a \times L_a^0$ equipped with the action
\begin{equation}\label{glgf}
(u,v)\cdot (\tu,\tv) = (u+\tu, v+\tv - \lambda u \tau(\tu) +\tau(\lambda) \tau(u)\tu).
\end{equation}
Then $H_0^\lambda$ is an algebraic $L_{2a}$-group and $j_\lambda: (u,v) \rightarrow (u, v - \lambda \tau(u)u)$ is an $L_{2a}$-group isomorphism of $H_0(L_a,L_{2a})$ onto $H_0^\lambda$. Let $H^\lambda = \Res_{F}^{L_{2a}} H_0^\lambda$.

Let $\gamma =-\frac{1}{2}\omega(\lambda)$. For $k\in \tilde\Gamma_a$, let $l= 2k+\frac{1}{e_a}$, 
and \[L_{a, k+\gamma}:=\{ u \in L_a\;| \; \omega(u) \geq k+\gamma\} \text{ and }  L_{a,l}^0:=\{ u \in L_a^0\;|\;\omega(u) \geq l\}.\]
Up to isomorphism, there exists a unique smooth affine $\fO_F$-group scheme $\cH_k^\lambda$ of finite type with generic fibre $H^\lambda$ and such that $\cH_k^\lambda(\fO_F) = L_{a,k+\gamma} \times L_{a,l}^0$ and a group law, which induces the group law \eqref{glgf} on the generic fibre (See Section 4.3.5 of \cite{BT2}). In more detail, let $\cL_{a,k+\gamma}$ and $\cL_{a,l}^0$ be the canonical $\fO_{L_{2a}}$-group schemes associated to $L_{a,k+\gamma}$ and $L_{a,l}^0$. Let $\cH_{0,k}^\lambda = \cL_{a,k+\gamma} \times \cL_{a,l}^0$. 
The map $L_a \times L_a \rightarrow L_a^0, (u,u') \rightarrow \lambda u \tau(\tu) - \tau(\lambda)\tau(u)\tu$ can be extended uniquely to a morphism $\cL_{a,k+\gamma} \times \cL_{a,k+\gamma} \rightarrow \cL_{a,l}^0$. Hence the group law can be extended to $\cH_{0,k}^\lambda$. Let $\cH_k^\lambda: = \Res_{F}^{\fO_{L_{2a}}} H_{0,k}^\lambda$. By transport of structure using $x_a \circ \Res_{F}^{L_{2a}} j_\lambda^{-1}$, we obtain the $\fO_F$-group scheme $\cU_{a,k}$. These definitions are extended to $k,l \in \bbR\backslash \{0\}$ in Section 4.3.8 of \cite{BT2}.

Using the isomorphism $v \rightarrow x_a(0,v)$ from $L_a^0 \rightarrow U_{2a}$, we obtain from the scheme $\cL_k^0$ (for $k \in \omega(L_a^0) \backslash 0$), an $\fO_F$-scheme whose generic fiber is $U_{2a}$ and denote it as $\cU_{2a,k}$ (see Section 4.3.7 of \cite{BT2} for further details). 
\end{enumerate}

\subsubsection{Construction of parahoric group schemes over $F$}\label{ParahoricLandvogt}
In this section, we recall the construction of parahoric group schemes, following \cite{La00}. Given $x \in \cA(S,F)$, let $f_x: \Phi(G,S) \rightarrow \bbR$ be the function $f_x(a) = a(x-x_0)$, where $x_0$ is the unique point arising from quasi-split descent as in Section \ref{ARS}. Let $U_{a,x} := U_{a, f_x(a)}$. Let $\cU_{a,x}$ be the smooth affine group scheme over $\fO_F$ with generic fiber $U_a$ and with $\cU_{a,x}(\fO_F) = U_{a,x}$ (as in Section \ref{RSGS}).  For $\Psi = \Phi^+(G,S)$ and $\Psi = \Phi^{-}(G,S)$, Proposition 3.3.2 of \cite{BT2}  gives a unique smooth affine $\fO_F$-group scheme $\cU_{\Psi,x}$ of finite type with generic fiber $U_{\Psi}$ and the property that for every good ordering of $\Psi^{red}$ (See Section 3.1.2 of \cite{BT2}), the $F$-isomorphism $\prod_{a \in \Psi} U_a \rightarrow U_{\Psi}$ can be extended to an $\fO_F$-isomorphism $\prod_{a \in \Psi} \cU_{a,x} \rightarrow \cU_{\Psi,x}.$

The parahoric subgroup $P_x$ is generated by $\cT(\fO_F)$ and the  $U_{a,x}$ for $ a \in \Phi(G,S)$  (with $\cT$ is as in Section \ref{CY}). One of the main results of \cite{BT2} is that there is a unique smooth affine $\fO_F$-group scheme $\cP_x$ with generic fiber $G$ and with $\cP_x(\fO_F) = P_x$. We recall the construction of $\cP_x$ from \cite{La00}. The idea is to put an $\fO_F$-birational group law on $\cU_{\Phi^+,x} \times \cT \times \cU_{\Phi^-,x}$ and invoke Artin-Weil theorem (see Chapters 5 and 6 of \cite{BLR}) to construct $\cP_x$. Let us first introduce some notation. Let $\cU^\pm_x = \cU_{\Phi^{\pm}(G,S), x}$ and let $\cX_x = \cU_x^- \cT\cU_x^+$. Since its generic fiber $\cX_x \times_{\fO_F} F = U^-TU^+$ is an open neighborhood of the 1-section of $G$, there exists a unique $F$-birational group law on the generic fiber of $\cX_x$. We want to extend this to $\cX_x$. Since $U^-TU^+$ and $U^+TU^-$ are both open neighborhoods of the 1-section of $G$, there exist $f \in F[U^-TU^+]$ and $f' \in F[U^+TU^-]$ such that $F[U^-TU^+]_f = F[U^+TU^-]_{f'}$.
Without loss of generality, we may assume that $f \in \fO_F[U^-TU^+]\backslash \pi \fO_F[U^-TU^+] $ and $f' \in \fO_F[U^+TU^-] \backslash \pi\fO_F[U^+TU^-]$. 
Proposition 5.16 of \cite{La00} shows that inside $F[U^-TU^+]_f = F[U^+TU^-]_{f'}$, we have $\fO_F[\cU_x^-\cT\cU_x^+]_f = \fO_F[\cU_x^+\cT\cU_x^-]_{f'}$. 
So we will identify $(\cU_x^-\cT\cU_x^+)_f = (\cU_x^+\cT\cU_x^-)_{f'}$ in the following. By Proposition 5.8 of \cite{La00}, we can identify $\cT\cU_x^+$ and $\cU_x^+\cT$ and hence also $\cT\cU_x^+\cU_x^-$ and $\cU_x^+\cT\cU_x^-$. In $\cX_x \times \cX_x = \cU_x^- \times (\cT\times \cU_x^+\times  \cU_x^-) \times \cT\cU_x^+$, we consider the open subscheme
\begin{align*}
\cU_x^- \times (\cU_x^+\times \cT\times  \cU_x^-)_f \times \cT\cU_x^+&= \cU_x^- \times ( \cU_x^- \times \cT \times \cU_x^+)_{f'} \times \cT\cU_x^+\\&\subset \cU_x^- \times \cU_x^- \times \cT \times \cU_x^+ \times \cT\cU_x^+\\
&=(\cU_x^- \times\cU_x^-) \times(\cT\times \cT) \times (\cU_x^+ \times\cU_x^+ )\\
&\xrightarrow{mult^3} \cU_x^- \times \cT\times \cU_x^+
\end{align*}
So we obtain a morphism $\cU_x^- \times (\cU_x^+\times \cT\times  \cU_x^-)_f \times \cT\cU_x^+ \rightarrow \cX_x.$ Since $\cX_x$ has irreducible fibers over $\fO_F$ and since $f \notin \pi\fO_F[\cU_x^-\cT\cU_x^+]$, we see that $(\cU_x^-\cT\cU_x^+)_f$ is $\fO_F$-dense in $\cX_x$ (that is, each of its fibers is Zariski dense in the corresponding fiber of $\cX_x$ - see Section 2.5 of \cite{BLR}), and hence $\cU_x^- \times (\cU_x^+\times \cT\times  \cU_x^-)_f \times \cT\cU_x^+$ is $\fO_F$-dense in $\cU_x^- \times (\cT\times \cU_x^+\times  \cU_x^-) \times \cT\cU_x^+=\cX_x \times \cX_x$. Hence we obtain an $\fO_F$-rational map $m:\cX_x \times \cX_x \rightarrow \cX_x$. By Proposition 5.16 of \cite{La00}, $m$ is an $\fO_F$-birational group law on $\cX_x$. Glue together the schemes $G$ and $\cX_x$ along $\cX_x \times_{\fO_F} F$ and denote it as $\cY_x$. As in Proposition 5.17 of \cite{La00}, the parahoric group scheme $\cP_x$ with group law $\bar{m}$, together with an open immersion $\cY_x \rightarrow \cP_x$ such that the restriction of $\bar{m}$ to $\cY_x$ is $m$, is obtained by applying Theorem 5.1 of \cite{BLR} to the scheme $\cY_x$. The generic fiber of $\cP_x$ is $G$. Let $\cF$ be a facet in $\cA(S,F)$. Then for $x, y \in \cF$, $P_x = P_y$. So we write $P_\cF$ for the parahoric subgroup attached to the facet $\cF$ and denote the underlying group scheme as $\cP_{\cF}$.
\subsection{Parahoric group schemes; \'{E}tale descent}\label{etaledescent}
Let $F$ be a non-archimedean local field and $\Fsh$ be the completion of the maximal unramified extension $\Fu (\subset F_s)$ of $F$. Let $G$ be a connected reductive group over $F$. 
By a theorem of Steinberg (recalled as Theorem \ref{SVT}), we know that $G_{\Fu}$ is quasi-split. 
Let $A$ be a maximal $F$-split torus in $G$. By Section 5 of \cite{BT2}, there is an $F$-torus $S$ that contains $A$ and is maximal $\Fu$-split. Note that $X_*(A) = X_*(S)^{\Gal(\Fsh/F)}$. 
Let $\cA(A,F)$ denote the apartment of $G$ with respect to $A$. Let $\cF_*$ be a facet in $\cA(A,F)$. We fix an algebraic closure $\bar\kappa_F$ of the residue field $\kappa_F$ and identify the Galois groups $\Gal(\Fsh/F)$ with $\Gal(\bar\kappa_F/\kappa_F)$. 
Let $\sigma$ denote the Frobenius element of $\Gal(\Fu/F)$ under this identification. Then we know that there is a $\sigma$-stable facet $\tilde\cF_*$ in $\cA(S, \Fu)$ such that $\tilde\cF_*^\sigma = \cF_*$ (see Chapter 5 of \cite{BT2}). 
Since $\tilde\cF_*$ is stable under the action of $\sigma$, the parahoric group scheme $\cP_{\tilde\cF_*}$ is also stable under the action of $\sigma$.
 In this case, the $\fO_{\Fsh}$-group scheme $\cP_{\tilde\cF_*}$ admits a unique descent to an $\fO_F$-group scheme with generic fiber $G$ (see Example B, Section 6.2, \cite{BLR}). 
 The affine ring of this group scheme is $\left(\fO_{\Fsh}[\cP_{\tilde\cF_*}]\right)^{\Gal(\Fsh/F)}$. This is the parahoric group scheme attached to the facet $\cF_*$ of $\cA(A,F)$. 

\section{Quasi-split forms over close local fields} 
Let $G_0$ be a split connected reductive group defined over $\bbZ$ with root datum $(R,\Delta)$.  For an extension $K/F$, let $G_{0,K}:= G_0 \times_\bbZ K$. 

Let $E(F, G_{0})$ be the of $F$-isomorphism classes of connected reductive $F$-algebraic groups $G$ with $G_{F_s}$ isomorphic to $G_{0,F_s}$. This is  in natural bijection with the Galois cohomology set $H^1(\Gamma_F, \Aut(G_{0,F_s}))$. We denote this map \begin{align}\label{FormsTheorem}
E(F, G_{0}) \rightarrow H^1(\Gamma_F, \Aut(G_{0,F_s})), [G] \rightarrow s_G.
\end{align}

\begin{lemma}\label{atmostmramified}
 Let $I_F$ be the inertia group of $F$ and $I_F^{m}$ denote the $m$-th higher ramification subgroup with upper numbering. Let $E(F, G_0)_m$ denote the set of $F$-isomorphism classes of $F$-forms  $G$ of $G_{0,F}$ such that there exists an atmost $m$-ramified finite extension $L \subset F_s$ (i.e. $\Gal(L/F)^m = 1$) with $G\times_F L \cong G_0 \times_\bbZ L$.  The bijection \eqref{FormsTheorem} induces a bijection between $E(F, G_0)_m$ and the cohomology set $H^1(\Gamma_F/I_F^m, (\Aut_{F_s}(G_{0,F_s}))^{I_F^m}).$ 
\end{lemma}
\begin{proof} Let $\Omega := (F_s)^{I_F^m}$. Then for every finite extension $F \subset L \subset F_s$, $L \hookrightarrow \Omega$ if and only if $\Gal(L/F)^m =1$ (see section 3.5 of \cite{Del84}). Further we know that $H^1(\Aut(\Omega/F), \Aut_\Omega(G_{0,\Omega}))$ classifies isomorphism classes of $F$-forms $[G]$ with $G \times_F \Omega \cong G_{0,F} \times_F \Omega$. Now simply note that $\Aut(\Omega/F) \cong \Gamma_F/I_F^m$ and $\Aut_\Omega(G_{0,\Omega}) =  (\Aut_{F_s}(G_{0,F_s}))^{I_F^m}$. 
\end{proof}
\subsection{Quasi-split forms} Let $(G_0, T_0, B_0, \{u_\alpha\}_{\alpha \in \tilde\Delta})$ be a pinned, split, connected, reductive $\bbZ$-group with based root datum $(R, \tilde\Delta)$ where $\{u_\alpha\}_{\alpha \in \tilde\Delta}$ is a splitting as in Section 3.2.2 of \cite{BT2}. Then $\Out(G_0)$ can be identified with the constant $\bbZ$-group scheme associated to the group $\Aut(R, \tilde\Delta)$. Consider the exact sequence
\[1 \rightarrow \Inn(G_0(F_s)) \rightarrow \Aut(G_{0, F_s}) \rightarrow \Aut(R, \tilde\Delta) \rightarrow 1.\]
Let $H = H(G_0, T_0, B_0, \{u_\alpha\}_{\alpha \in \tilde\Delta})$ be the subgroup of $\Aut(G_{0,F_s})$ consisting of all $a$ such that $a(B_0) = B_0$, $a(T_0) = T_0$ and $\{a \circ u_\alpha \;|\; \alpha \in \tilde\Delta\} = \{u_\alpha \;|\; \alpha \in \tilde\Delta\}$. Then $H \hookrightarrow \Aut(G_{0,F_s}) \rightarrow \Aut(R, \tilde\Delta)$ is an isomorphism and $\Aut(G_{0,F_s}) \cong H \ltimes \Inn(G_0(F_s))$. Hence the natural map  $ H^1(\Gamma_F, \Aut(G_{0,F_s})) \rightarrow H^1(\Gamma_F, \Aut(R, \tilde\Delta))$ has a section given by 
\[q: H^1(\Gamma_F, \Aut(R,\tilde\Delta)) \xrightarrow{\cong} H^1(\Gamma_F, H) \rightarrow H^1(\Gamma_F, \Aut(G_{0,F_s})).\]
We now recall the following well-known theorem (see \cite{Conrad}, Section 7.2). 
\begin{theorem}\label{qsforms} Let $[G] \in E(F, G_0)$. Then $s_G$ lies in the image of $q: H^1(\Gamma_F, \Aut(R,\tilde\Delta)) \rightarrow H^1(\Gamma_F, \Aut(G_{0,F_s}))$ if and only if $G$ is quasi-split over $F$, that is, it has a Borel subgroup defined over $F$. 
\end{theorem}
Let $E_{qs}(F, G_0):=\{ [G] \in E(F,G_0)\;|\; s_G \in Im(q)\}$ and $E_{qs}(F, G_0)_m = E_{qs}(F, G_0) \cap E(F, G_0)_m$. Since $G_0$ is $F$-split, the action of $\Gamma_F$ on $(G_0, B_0, T_0)$ is trivial. Hence 
$Z^1(\Gamma_F, \Aut(R,\tilde\Delta)) = \Hom(\Gamma_F, \Aut(R,\tilde\Delta))$.
\begin{lemma}\label{QSF} We have the following:
\begin{enumerate}[leftmargin = *]
\item The class $[G] \in E_{qs}(F, G_0)_m$ if and only if $s_G $ lies in the image of \[q: H^1(\Gamma_F/I_F^m, \Aut(R,\tilde\Delta)^{I_F^m}) \rightarrow H^1(\Gamma_F/I_F^m, \Aut(G_{0,F_s})^{I_F^m})\]
\item The isomorphism $\psi_m: \Tr_m(F) \xrightarrow{\cong} \Tr_m(F')$ induces an isomorphism
\[ \fQ_m: H^1(\Gamma_F/I_F^m, \Aut(R,\tilde\Delta)) \xrightarrow{\cong} H^1(\Gamma_{F'}/I_{F'}^m, \Aut(R,\tilde\Delta))\]
and
\[ \fQ_m^c: Z^1(\Gamma_F/I_F^m, \Aut(R,\tilde\Delta)) \xrightarrow{\cong} Z^1(\Gamma_{F'}/I_{F'}^m, \Aut(R,\tilde\Delta))\]
\item The isomorphism $\psi_m$ induces a bijection $E_{qs}(F, G_0)_m \rightarrow E_{qs}(F', G_0)_m, [G] \rightarrow [G']$, where $s_{G'}  = q' \circ \fQ_m(s_G)$.
\end{enumerate}
\end{lemma}
\begin{proof} This is clear from Lemma \ref{atmostmramified} and Theorem \ref{qsforms}. 
\end{proof}

As noted in Lemma \ref{QSF}, $Z^1(\Gal(\Omega/F), \Aut(R,\tilde\Delta)) = Hom(\Gal(\Omega/F), \Aut(R,\tilde\Delta))$ since $G_0$ is split. Let us fix $s \in Z^1(\Gal(\Omega/F), \Aut(R, \tilde\Delta)) \cong Z^1(\Gal(\Omega/F), H)$.
 Let $(G, \phi)$ be a pair of be a quasi-split connected reductive group over $F$ and $\phi: G_0\times_{\bbZ} \Omega \rightarrow G \times_F \Omega$ an $\Omega$-isomorphism such that the Galois action on $G(F_s)$ is given by $s$. We may and do assume that there is a finite Galois atmost $m$-ramified extension $K$ of $F$ over which $\phi$ is defined, that is, that $s \in Z^1(\Gal(K/F),\Aut(R, \tilde\Delta))$.

    More precisely, with $*_F$ denoting the Galois action on $G(K)$, we have \[\gamma*_F \phi(x) = \phi(s(\gamma)(\gamma\cdot x))\] for $\gamma \in \Gal(K/F)$ and $x \in G_0(K)$. 
 Then $\phi(T_0) = T$ is a maximal torus of $G$ defined over $F$ and $\phi(B_0) = B$ is a Borel subgroup of $G$ containing $T$ and defined over $F$. Let $s' \in Z^1(\Gal(K'/F'), \Aut(R, \tilde\Delta)) $ as in Lemma \ref{QSF}. Here $K'/F'$ is determined by $K/F$ via $\Del_m$. 
 Let $(G', \phi')$ be a pair of quasi-split connected reductive group over $F'$ and $\phi': G_0 \times_\bbZ K' \rightarrow   G'\times_{F'} K'$ such that $\gamma'*_{F'} \phi'(x') = \phi(s'(\gamma')(\gamma'\cdot x'))$, where $\gamma' =\Del_m(\gamma)$. 
 Then $\phi'(T_0) = T'$ and $\phi'(B_0) = B'$ are defined over $F'$. 
Note that $X_*(T) \cong X_*(T_0) \cong X_*(T')$ and $X^*(T) \cong X^*(T_0) \cong X^*(T')$ via $\phi$ and $\phi'$. 

Recall the notation of Chai-Yu: 
$(\fO_F, \fO_K, \Gamma_{K/F}) \equiv_{\psi_m, \gamma} (\fO_{F'}, \fO_{K'}, \Gamma_{K'/F'}) \;(level\;\; m)$ from Section \ref{CY}. 

We write \[(\fO_F, \fO_K, \Gamma_{K/F}, H) \equiv_{\psi_m,\gamma, \fQ_m^c} (\fO_{F'}, \fO'_{K'}, \Gamma_{K'/F'}, H') \;(level\;\; m)\] to mean $(\fO_F, \fO_K, \Gamma_{K/F}) \equiv_{\alpha,\beta} (\fO_{F'}, \fO'_{K'}, \Gamma_{K'/F'}) \;(level\;\; m)$, $H$ and $H'$ arise from the same $\bbZ$-pinned group $(G_0, B_0, T_0, \{u_\alpha\}_{\alpha \in \tilde\Delta})$, and the $F$-quasi-split data $(G, B, T)$ with cocycle $s$ corresponds to the $F'$-quasi-split data $(G', B', T')$ with cocycle $s'$ via $\fQ_m^c$ as in Lemma \ref{QSF} $(b)$ (but applied to $K$ and $K'$ respectively). 
To abbreviate notation we will write \textit{congruence data $D_m$} to mean
\[D_m: (\fO, \fO_K, \Gamma_{K/F}, H) \equiv_{\psi_m, \gamma, \fQ_m^c} (\fO_{F'}, \fO'_{K'}, \Gamma_{K'/F'}, H') \;(level\;\; m).\] 
\begin{lemma}\label{charQSCLF}
 The congruence data $D_m$ induces isomorphisms $X^*(T)^{\Gal(\Omega/F)} \cong X^*(T')^{\Gal(\Omega'/F')}$, $X_*(T)^{\Gal(\Omega/F)} \cong X_*(T')^{\Gal(\Omega'/F')}$, $X^*(T)_{\Gal(\Omega/F)} \cong X^*(T')_{\Gal(\Omega'/F')}$, and  $X_*(T)_{\Gal(\Omega/F)} \cong X_*(T')_{\Gal(\Omega'/F')}$.
 \end{lemma}
 \begin{proof} We know that $\gamma *_F(\phi(x)) = \phi(s(\gamma)(\gamma\cdot x))$ where $s(\gamma) = \phi^{-1} \circ \gamma(\phi)$ takes values in $H = H(G_0, T_0, B_0, \{u_\alpha\}_{\alpha \in \tilde\Delta})$. We similarly have $*_{F'}$. This action induces the action on $X_*(T)$ as follows: \[\gamma *_F(\phi \circ \lambda) = \phi \circ (s(\gamma)(\lambda))\] where $\gamma \in \Gal(\Omega/F)$ and $\lambda \in X_*(T_0)$, where we now view $s(\gamma)$ as an element of $\Aut(R, \tilde\Delta))$.  By definition $s(\gamma)(\lambda) = s'(\gamma')(\lambda)$ where $\gamma' = \Del_m(\gamma)$. Hence $\gamma'*_{F'}(\phi' \circ \lambda) = \phi' \circ s'(\gamma')(\lambda)$. Now, $X_*(T)^{\Gal(\Omega/F)} = \{\phi \circ \lambda \;|\; s(\gamma)(\lambda) = \lambda\}$. The lemma is now clear. 
 \end{proof}
 \section{Congruences of parahoric group schemes; Quasi-split descent}\label{QS}
\subsection{Apartment over close local fields} In this section, we additionally assume that $F$ is strictly Henselian.
We begin with the following lemma.
\begin{lemma}\label{maxFsplit}
Let $T$ as above and let $S$ be the maximal split subtorus of $T$. Then $S$ is maximal $F$-split and $Z_G(S) = T$. 
\end{lemma}
\begin{proof}
Let $S \subset \tS$ with $\tS$ maximal $F$-split. Since $G$ is quasi-split over $F$, $\tT= Z_G(\tS)$ is a maximal torus in $G$ and we can assume that $\tT \subset \tB$, with $\tB$ defined over $F$. Then $B$ and $\tB$ are $G(F)$-conjugate,  which implies that $T$ and $\tT$ are $G(F)$-conjugate. But conjugation by an element of $G(F)$ will preserve the split and anisotropic components of $T$, which implies that $S$ and $\tS$ are $G(F)$-conjugate, which forces $S = \tS$ to be maximal $F$-split. It is now clear that $Z_G(S) =T$. 
\end{proof}
\begin{remark} The torus $S^{der} := S \cap G^{der}$ is a maximal $F$-split torus of $G^{der}$ contained in $T^{der} := T \cap G^{der}$. 
\end{remark}
\subsubsection{Compatibility of Chevalley-Steinberg systems} \label{CSCLF} Recall that we have fixed a $\bbZ$-pinning $\{u_\alpha\}_{\alpha \in \Delta}$ of $G_0$. This, via the Galois action given by the cocycles $s$ and $s'$, gives rise to a Steinberg splitting $\{x_\alpha\}_{\alpha \in \Delta}$ of $G$ and a Steinberg splitting $\{x_{\alpha'}'\}_{\alpha' \in \Delta'}$ of $G'$ respectively. Let $\Phi_m: \Phi(G, T) \xrightarrow{\cong} \Phi(G',T')$ (since both are isomorphic to $\Phi(G_0,T_0)$). This isomorphism is $\Del_m$-equivariant. Note that with $\gamma \in \Gal(\Omega/F)$ and $\gamma' = \Del_m(\gamma)$, we have that $x_{\gamma (\alpha)} = \gamma \circ x_\alpha \circ \gamma^{-1}$ and $x_{\gamma'(\alpha')}' = \gamma' \circ x_{\alpha'}' \circ \gamma'^{-1}$ where $\alpha' = \Phi_m(\alpha)$. The $\{x_\alpha\}_{\alpha \in \Delta}$ and $\{x'_{\alpha'}\}_{\alpha' \in \Delta'}$ each extend to Chevalley-Steinberg systems on $G$ and $G'$ respectively and continue to have the  compatibility with $\Del_m$ in the sense described above. 

We define
\[e_F := \begin{cases}e_{F/\bbQ_2} =\omega_F(2) & \text{ if } \Char(F)=0 \text{ and residue char}(F)=2 \\
\infty &  \text{ otherwise.}
\end{cases}\]
We prove the following refinement of Lemma 4.3.3 of \cite{BT2} when the residue characteristic of $F$ is 2, using the additional hypothesis that the extension $K/F$ splitting $G$ is atmost $m$-ramified. 

\begin{lemma}\label{reschar2} Let $m \geq 1$ and let $F$ be of residue characteristic 2 with $e_F \geq m$. Let $G,B,T$ as above, where $G$ splits over $K$ with $\Gal(K/F)^m=1$. Assume that $a,2a \in \Phi(G,S)$. Consider the separable quadratic extension $L_a/L_{2a}$ inside $K$. Let $e_a = e_{L_a/F}, e_{2a} = e_{L_{2a}/F}$.  There exists $t \in L_a$ with $L_a = L_{2a}[t]$ and the coefficients $A,B \in L_{2a}$ of the equation $t^2+A t+B =0$ satisfied by $t$ have the following properties.
\begin{enumerate}[leftmargin=*]
\item $\omega(B) = 0$ or $B$ is a uniformizer of $L_{2a}$.
\item $\omega(B) \leq \omega(A)< \frac{m}{2} +\frac{1}{e_a}$. 
\end{enumerate}
In particular $A \neq 0$. 
\end{lemma}
\begin{proof} By lemma 4.3.3 (ii) of \cite{BT2}, (a) holds, and $A=0$ or $\omega(B)\leq \omega(A) <\omega(2)$ or $0<\omega(B)\leq \omega(A) =\omega(2)$. Since $\Gal(K/F)^m = \Gal(K/F)_{\psi_{K/F}(m)} =1$ where $\psi_{K/F}$ denotes the inverse of the Herbrand function (See Chapter 4 of \cite{Ser79}), we have
\[\Gal(K/L_{2a})^{\psi_{L_{2a}/F}(m)} = \Gal(K/L_{2a})_{\psi_{K/F}(m)} = \Gal(K/L_{2a}) \cap \Gal(K/F)_{\psi_{K/F}(m)} =1.\]  
This implies that $\Gal(L_a/L_{2a})^{\psi_{L_{2a}/F}(m)} =1.$
 Using the equivalence of (ii) and (iv) of Lemma A.6.1 of \cite{Del84}, we see that
\begin{align}\label{inequality1}
\omega(\tau(t)-t)< \frac{\psi_{L_{2a}/F}(m)+1}{2e_{2a}} = \frac{\psi_{L_{2a}/F}(m)+1}{e_{a}}.
\end{align}

It is easy to see from the definition that $\psi_{L_{2a}/F}(m) \leq m \cdot e_{2a}$. 
Hence \[\omega(\tau(t) -t) < \frac{m}{2} +\frac{1}{e_a}.\] 

Now, $\omega(A) = \omega(\tau(t) +t) \geq \min(\omega(\tau(t)-t), \omega(2t))$,
and $\omega(2t) = \omega(2)+\omega(t) = e_F+\frac{1}{e_{a}}$. Since $e_F\geq m >m/2$, we see that
\begin{align}
\omega(A)= \min(\omega(\tau(t)-t), \omega(2t)) = \omega(\tau(t)-t) <\frac{m}{2} +\frac{1}{e_a}
\end{align}
and in particular, $A\neq 0$. 

Note that when the characteristic of $F$ is 2, the claim that $A \neq 0$ simply follows from the fact that the extension $L_a/L_{2a}$ is separable. 
\end{proof}
 
 \begin{proposition}\label{ACLF}
 Let $G$, $T$ and $B$ as in the preceding paragraph. Let $m \geq 1$ and let $F,F'$ be such that $e_F, e_{F'} \geq m$. The congruence data $D_m$ induces a simplicial isomorphism $\cA_m: \cA(S, F) \rightarrow \cA(S', F')$, where $(G',B', T')$ corresponds to the triple $(G, B, T)$ as above and $S$ (resp. $S'$) is the maximal split subtorus of $T$ (resp. $T')$ which is maximal $F$-split (resp. $F'$-split) by Lemma \ref{maxFsplit}. Furthermore, with $W^e$ as in Section \ref{ARS}, we also have a group isomorphism $W^{e} \cong W^{e'}$. 
\end{proposition}
\begin{proof}
The reduced apartment $\cA(S,F)$ is an affine space under $\displaystyle{X_*(S^{der}) \otimes_\bbZ \bbR}$. Using Lemma \ref{charQSCLF}, we see that $D_m$ induces a unique bijection $\cA_m: \cA(S,F) \rightarrow \cA(S',F')$ such that $x_0 \rightarrow x_0'$ (where $x_0, x_0'$ are as in Section \ref{ARS} arising from Chevalley-Steinberg systems chosen compatibly as in Section \ref{CSCLF}).

It remains to observe that $\cA_m$ is a simplicial isomorphism. 
Recall that the elements of $\Phi(G,S)$ are restrictions to $S$ of the elements of $\Phi(G,T)$ and two elements of $\Phi(G,T)$ restrict to the same element of $\Phi(G,S)$ if and only if they lie in the same $\Gal(K/F)$-orbit. Further, with $\tilde{\Delta}$ denoting a base of $\Phi(G, T)$, the elements $\alpha|_S, \alpha \in \tilde{\Delta}$ form a base $\Delta$ of $\Phi(G,S)$. Let $\Phi_m: \Phi(G, T) \xrightarrow{\cong} \Phi(G',T')$ (since both are isomorphic to $\Phi(G_0,T_0)$). This isomorphism is $\Del_m$-equivariant. Hence the obvious map $\Phi(G,S) \rightarrow \Phi(G',S'),\; \alpha|_S \rightarrow \Phi_m(\alpha)|_{S'}$, which we also denote as $\Phi_m$, is an isomorphism of the relative root systems (In more detail, since $S$ and $S'$ have the same rank, we have a isomorphism of $\bbR$-vector spaces $X^*(S) \otimes_\bbZ \bbR \rightarrow X^*(S') \otimes_{\bbZ} \bbR$. Further, we have a bijection between $\Delta \rightarrow \Delta'$; this is because suppose $\Phi_m(\alpha)|_{S'}  = \Phi_m(\beta)|_{S'}$, then there is $\eta' \in \Gal(\Omega'/F')$ with $\eta'\cdot \Phi_m(\alpha) = \Phi_m(\beta)$. Then $\eta \cdot \alpha  =\beta$ where $\eta' = \Del_m(\eta)$. Finally note that $\langle \Phi_m(\alpha)|_{S'}, \Phi_m(\beta)|_{S'} \rangle =\langle \Phi_m(\alpha), \Phi_m(\beta) \rangle  = \langle \alpha, \beta \rangle = \langle \alpha|_{S}, \beta|_{S} \rangle$). 

 The vanishing hyperplanes with respect to the affine roots $\Phi^{af}(G,S)$ gives the simplicial structure on $\cA(S,F)$. Recall that
\[\Phi^{af}(G,F) = \{\psi:\cA(S,F) \rightarrow \bbR \;|\; \psi(\cdot ) = a(\cdot -x_0)+l, a \in \Phi(G,S), \; l \in \tilde\Gamma_a\}.\]
For any $a \in \Phi(G,S)$, let $a' =\Phi_m(a)$.  Let $L_{a'} \subset K'$ denote splitting extension of the root $a'$ obtained by $\Del_m$. Since $F$ is strictly Henselian, the extensions $L_a/F$ and $L_{a'}/F'$ are totally ramified. To prove that the bijection $\Phi_m$ extends to a bijection $\Phi^{af}_m:\Phi^{af}(G,F)\rightarrow \Phi^{af}(G',F')$ making $\cA_m$ a simplicial isomorphism, we simply have to observe that for each $a \in \Phi(G,S)$, $\tilde\Gamma_a = \tilde\Gamma_{a'}$.  By Section 4.3.4 of \cite{BT2}, we have the following:
\begin{enumerate}[wide, labelwidth=!, labelindent=0pt]
\item[\textbf{Case} I.] Suppose $a \in \Phi^{red}(G,S), 2a \notin \Phi(G,S)$. Then $\Gamma_a = \tilde\Gamma_a = \frac{1}{e_a}\bbZ$. 
 \item[\textbf{Case} II.] Suppose $a, 2a \in \Phi(G,S)$.
 \begin{enumerate}
 \item Suppose $L_a/L_{2a}$ is ramified and the residue characteristic of $F$ is not 2. Then  \[\tilde\Gamma_a = \frac{1}{e_a}\bbZ\text{ and }\tilde\Gamma_{2a} = \frac{1}{e_a}+\frac{1}{e_{2a}}\bbZ.\] 
\item Suppose $L_a/L_{2a}$ is ramified and the residue characteristic of $F$ is 2. By Lemma \ref{reschar2}, $A \neq 0$. Then \[\tilde\Gamma_a = \frac{1}{2e_a} +\frac{1}{e_a}\bbZ \text{ and } \tilde\Gamma_{2a}=\frac{1}{e_{2a}}\bbZ.\]
\end{enumerate}
\end{enumerate}
Since $e_a = e_{a'}, e_{2a} = e_{2a'}$, and the valuations $\omega$ and $\omega'$ are normalized so that  $\omega(F) = \omega'(F') = \bbZ$, we have $\tilde\Gamma_a = \tilde\Gamma_{a'}$ for all $a \in \Phi(G,S)$.
\end{proof}

\subsection{Congruences of parahoric group schemes; Strictly Henselian case}
In this section, we additionally assume that $F$ is strictly Henselian.
\begin{theorem}\label{QSP} Let $m \geq 1$ and let $F$ and $F'$ be such that $e_F,e_{F'} \geq 2m$.  Let $l$ be as in Lemma \ref{torusCLF} and let $D_l$ and $G, S, T, B$ as in the beginning of this section. Let $\cF \in \cA(S, F)$ and $\cF' = \cA_m(\cF)$ as in Lemma \ref{ACLF}. Let $\cP_\cF$ be the parahoric group scheme  over $\fO_F$ attached to $\cF$ by Bruhat-Tits, and let $\cP_{\cF'}$ be the group scheme attached to $\cF'$ over $\fO_{F'}$. Then the congruence data $D_l$ induces an isomorphism of group schemes
\begin{align*}
\tp_m: \cP_{\cF} \times_{\fO_F} \fO_F/\fp^m_{F} \rightarrow \cP_{\cF'} \times_{\fO_{F'}} \fO_{F'}/\fp_{F'}^m \times_{\psi_m^{-1}} \fO_F/\fp_F^m.
\end{align*}
In particular, $\cP_\cF(\fO_F/\fp_F^m) \cong \cP_{\cF'}(\fO_{F'}/\fp_{F'}^m)$ as groups. 
\end{theorem}
To prove this theorem, we will study the reduction of root subgroup schemes mod $\fp_F^m$ and prove that they are determined by congruence data, use the result of Chai-Yu that the reduction of the  N\'eron model of the torus in determined by congruence data, study the reduction of $\fO_F$-birational group laws in Section \ref{ParahoricLandvogt}, and invoke the Artin-Weil theorem to obtain the corresponding result for parahoric group schemes in Section \ref{QSPSH}. 

The following lemma is easy. 

\begin{lemma} Let $M$ be a free $\fO_F$-module of finite type and let $A = \Sym_{\fO_F}(M^\vee)$ be the symmetric algebra of $M^\vee$, where $M^\vee: = Hom_{\fO_F}(M, \fO_F)$. Then \[A \otimes_{\fO_F} \fO_F/\fp_F^m \cong \Sym_{\fO_F/\fp_F^m} (M^\vee \otimes_{\fO_F} \fO_F/\fp_F^m) \cong \Sym_{\fO_{F}/\fp_{F}^m} (\Hom_{\fO_F/\fp_F^m}(M \otimes_{\fO_F} \fO_F/\fp_F^m, \fO_F/\fp_F^m)).\] 
\end{lemma}

\begin{lemma}\label{rootgroupschemes} Let $m \geq 1$, let $F$ and $F'$ be such that $e_F,e_{F'}\geq 2m$ and let $D_m$ as before. Let $a \in \Phi(G,S)$ and $k \in \bbR$. Let $\cU_{a,k}$ be the $\fO_F$-group scheme in Section \ref{RSGS}. Let $a' =\Phi_m(a) \in \Phi(G',S')$ and let $\cU_{a',k}'$ be the $\fO_{F'}$-group scheme in Section \ref{RSGS}. Then the congruence data $D_m$ induces an isomorphism of group schemes 
\[\cU_{a,k} \times_{\fO_F} \fO_F/\fp_F^m \cong \cU_{a',k} \times_{\fO_{F'}} \fO_{F'}/\fp_{F'}^m \times_{\psi_m^{-1}} \fO_F/\fp_F^m.\]
In particular,
\[\cU_{a,k}(\fO_{F}/\fp_{F}^m) \cong \cU_{a',k}(\fO_{F'}/\fp_{F'}^m).\]
\end{lemma}
\begin{proof} We will  stick to the notation in Section \ref{RSGS}. 
\begin{enumerate}[wide, labelwidth=!, labelindent=0pt]
  \item[\textbf{Case} I.] Suppose $a\in \Phi^{red}(G,S), 2a \notin \Phi(G,S)$. The affine ring representing $\cU_{a,k}$ is isomorphic to  $\Sym_{\fO_F} L_{a,k}^\vee$. 
Note that $L_{a,k} = \fp_{L_a}^{\lceil k/e\rceil}$. Since $\fp_{L_a}$ is a free $\fO_F$-module of rank equal to $[L_a:F]$, it is clear that the data $D_m$ induces an isomorphism of  $L_{a,k} \otimes_{\fO_F} \fO_{F}/\fp_F^m$ and $L_{a',k} \otimes_{\fO_{F'}} \fO_{F'}/\fp_{F'}^m$ and we are done by the previous lemma.
\item[\textbf{Case} II.] Suppose $a, 2a \in \Phi(G,S)$.  Since $F$ is strictly Henselian, the extension $L_a/L_{2a}$ is totally ramified. Let $L_a = L_{2a}(t)$, where $t^2+At+B = 0$ with $A,B$ satisfying Lemma 4.3.3 of \cite{BT2}. When the
\begin{itemize}[leftmargin=*]
\item residue characteristic of $F$ is not 2, we take $\lambda = 1/2$ (See Lemma 4.3.3 (ii) of \cite{BT2}),
\item residue characteristic of $F$ is 2, we take $\lambda = tA^{-1}$ (using  Lemma 4.3.3 (ii) of \cite{BT2} and Lemma \ref{reschar2}). 
\end{itemize}
Then the affine ring representing the scheme $\cH_0^{\lambda}$ is $\Sym_{\fO_{L_{2a}}} L_{a,k+\gamma}^\vee \otimes_{\fO_{L_{2a}}}  \Sym_{\fO_{L_{2a}}} (L_{a,l}^0)^\vee \cong \Sym_{\fO_{L_{2a}}}( (L_{a,k+\gamma}  \times L_{a,l}^0)^\vee)$, where $l = 2k+\frac{1}{e_a}$.
We describe $L_{a,l}^0$.
\begin{enumerate}
\item If the residue characteristic of $F$ is not 2, then using that $\omega(2) = 0$ in Lemma 4.3.3 of \cite{BT2}, we see that $A = 0$. Then $L_a^0 = \{ x \in L_a \;|\; \tau(x) +x = 0\}=\{ yt\;|\; y \in L_{2a}\}$ 
and \[L_{a,l}^0=\{ yt\;|\; y \in L_{2a},\; \omega(yt)\geq l\} = \{yt\;|\; y \in L_{2a}, \omega(y) \geq 2k\}.\]
\item If the residue characteristic of $F$ is 2, then
\begin{enumerate}
\item if $\Char(F)=2$, then $L_a^0 = L_{2a}$ and $L_{a, l}^0 = \{y \in L_{2a}\;|\; \omega(y) \geq l\}.$
\item if $\Char(F) = 0$, then $L_a^0 = \{y(1-2tA^{-1})\;|\; y \in L_{2a}\}$. By Lemma \ref{reschar2}, we have \[\omega(2tA^{-1}) = e_F+\frac{1}{e_a} -\omega(A) > e_F-\frac{m}{2}\geq m\]
since $e_F \geq 2m$. Hence $1-2tA^{-1} \in 1+\fp_{L_{a}}^{me_{a}}$, and $L_{a, l}^0 = \{y(1-2tA^{-1})\;|\; y \in L_{2a}, \omega(y) \geq l\}.$
\end{enumerate}
\end{enumerate}

 Let $L_{a'}\subset \Omega'$ be obtained from $L_a$ via the Deligne isomorphism $\Del_m$. Then $L_{a'}$ is the splitting extension of the root $a'$ (and similarly we obtain $L_{2a'}$). We may and do assume that
 $L_{a'} = L_{2a'}(t')$, where $t'^2+A't'+B' =0$, with $A',B'$ satisfying
 \begin{itemize}[leftmargin=*]
 \item $\omega(A) = \omega'(A')$ and $A\mod \fp_{L_{2a}}^{me_{2a}} \xrightarrow{\psi_m} A' \mod \fp_{L_{2a'}'}^{me_{2a}}$
 \item $\omega(B) = \omega'(B')$ and $B\mod \fp_{L_{2a}}^{me_{2a}} \xrightarrow{\psi_m} B' \mod \fp_{L_{2a'}'}^{me_{2a}}.$
 \end{itemize}
 Then $t \mod \fp_{L_{a}}^{me_{a}} \xrightarrow{\psi_m} t' \mod \fp_{L_{a'}'}^{me_{a}}.$
 It is now easy to check that the map $\psi_m$ induces isomorphisms
 \begin{align*}
 L_{a, k+\gamma}  \otimes_{\fO_{L_{2a}}} \fO_{L_{2a}}/\fp_{L_{2a}}^{me_{2a}} &\cong L_{a',k+\gamma}^0 \otimes_{\fO_{L_{2a'}'}} \fO_{L_{2a'}'}/\fp_{L_{2a'}'}^{me_{2a}}\\
 L_{a,l}^0  \otimes_{\fO_{L_{2a}}} \fO_{L_{2a}}/\fp_{L_{2a}}^{me_{2a}} &\cong L_{a',l}^0 \otimes_{\fO_{L_{2a'}'}} \fO_{L_{2a'}'}/\fp_{L_{2a'}'}^{me_{2a}}.
 \end{align*}
In the above, we have used that when the residue characteristic of $F$ is 2,  $1-2tA^{-1} \equiv 1 \mod \fp_{L_a}^{me_a}$.
 Consequently, $D_m$ induces an isomorphism of the reduction of the respective affine rings mod $\fp_{L_{2a}}^{me_{2a}}$. To see that this is an isomorphism of group schemes, we observe that reducing the map 
\begin{align*}
j: L_{a,k} \times L_{a, l}^0 \times L_{a,k} \times L_{a,l}^0 &\rightarrow L_{a,k} \times L_{a,l}^0\\
((x,y) , (x',y')) &\rightarrow (x+x', y+y' -\lambda x \tau(x') +\lambda x'\tau(x))
\end{align*}
mod $\fp_{L_{2a}}^{me_{2a}}$ is $\psi_m$-equivariant. Finally $\cH^{\lambda} = \Res_{\fO_F}^{\fO_{L_{2a}}} \cH_0^{\lambda}$ and the result now follows from \cite{BLR}, Page 192. 
\end{enumerate}
 The lemma for $\cU_{2a,k}$ follows using that
\[L_{a,k}^0  \otimes_{\fO_{L_{2a}}} \fO_{L_{2a}}/\fp_{L_{2a}}^{me_{2a}} \cong L_{a',k}^0 \otimes_{\fO_{L_{2a'}'}} \fO_{L_{2a'}'}/\fp_{L_{2a'}'}^{me_{2a}}\]
and \cite{BLR}, Page 192.
\end{proof}
The following corollary is an obvious consequence of the previous lemma.
\begin{corollary}\label{unipotentCLF}  With assumptions of Lemma \ref{rootgroupschemes}, and with $\cF' = \cA_m(\cF)$ where $\cF$ is a facet in $\cA(S,F)$, let $\cU_{a,\cF}$ (resp. $\cU_{a',\cF'}$) be the smooth root subgroup scheme over $\fO_F$ (resp.  $\fO_{F'}$) as in Section \ref{ParahoricLandvogt}.  The congruence data $D_m$ induces an isomorphism
\[ \cU_{a,\cF} \times_{\fO_F} \fO_F/\fp_F^m \cong \cU_{a, \cF'} \times_{\fO_{F'}} \fO_{F'}/\fp_{F'}^m \times_{\psi_m^{-1}} \fO_F/\fp_F^m.\] In particular, $\cU_{a,\cF}(\fO_F/\fp_F^m) \cong \cU_{a', \cF'}(\fO_{F'}/\fp_{F'}^m)$ as groups. 
\end{corollary}

\subsubsection{Proof of Theorem \ref{QSP}}\label{QSPSH}
For a scheme $X$ defined over a local ring $R$ with maximal ideal $\fm$, we will denote $X^{(m)}:= X \times_{R} R/\fm^m$. Let $l$ be as in Lemma \ref{torusCLF}.  We want to prove that $D_l$ induces an isomorphism of $\fO_F/\fp_F^m$-group schemes $\cP_\cF^{(m)} \cong \cP_{\cF'}^{(m)} \times_{\psi_m^{-1} } \fO_F/\fp_F^m$. Let $\cX_\cF, \cX_{\cF'}$ be as in Section \ref{ParahoricLandvogt}.  Let $m^{(m)}$ be the $\fO_F/\fp_F^m$-birational group law on $\cX_\cF^{(m)}$ and similarly $n^{(m)}$  on $\cX_{\cF'}^{(m)}$.  Note that via $D_l$, we also have that
\[(\fO_F,\fO_K, \Gamma_{K/F}, \Lambda) \equiv_{\psi_e, \gamma, \lambda} (\fO_{F'},\fO_{K'}, \Gamma_{K'/F'}, \Lambda')\; (level\;\; l)\] 
as in the notation of Chai-Yu of Section \ref{CY}, where $\Lambda = X_*(T), \Lambda' = X_*(T')$; so the result of Lemma \ref{torusCLF} holds. We know by Lemmas \ref{torusCLF} and \ref{unipotentCLF} that 
\begin{align}\label{schemeiso}
\cX_\cF^{(m)} \cong \cX_{\cF'}^{(m)} \times_{\psi_m^{-1} } \fO_F/\fp_F^m
\end{align}
as $\fO_F/\fp_F^m$-schemes.  Further, by these lemmas, we also have that the $\fO_F/\fp_F^m$-birational group laws $n^{(m)} \times_{\psi_m^{-1} } \fO_F/\fp_F^m$ and $m^{(m)}$ on $\cX_\cF^{(m)}$ are equivalent. Since $\cY_\cF$ is the $\fO_F$-scheme obtained by gluing $G$ and $\cX_\cF$ along $\cX_\cF\times_{\fO_F}F$, we have that $\cY_\cF^{(m)}$ is isomorphic to $\cX_\cF^{(m)}$ as $\fO_F/\fp_F^m$-schemes. Now, $\cP_\cF^{(m)}$ with group law $\bar{m}^{(m)}$, and $\cP_{\cF'}^{(m)} \times_{\psi_m^{-1} } \fO_F/\fp_F^m$ with group law $\bar{n}^{(m)}\times_{\psi_m^{-1} } \fO_F/\fp_F^m$, are both smooth, separated $\fO_F/\fp_F^{m}$- group schemes that are faithfully flat and of finite type. Recall that the restriction of $\bar m$ to $\cY_\cF$ is $m$, and similarly for $\bar n$. Hence the group laws $\bar{m}^{(m)}$ and $\bar{n}^{(m)}\times_{\psi_m^{-1} } \fO_F/\fp_F^m$ have the same restriction to $\cY_\cF^{(m)}$. Following the the proof of uniqueness of Artin-Weil theorem (see Proposition 3, Section 5.1 of \cite{BLR}), we obtain that the group schemes $\cP_\cF^{(m)}$ and $\cP_{\cF'}^{(m)} \times_{\psi_m^{-1} } \fO_F/\fp_F^m$ are isomorphic.\qed

\subsection{Congruences of parahoric group schemes; Descending from $G_{ \Fsh}$ to $G_{F}$} In this section, $F$ denotes a non-archimedean local field and $\Fsh$ denotes the completion of the maximal unramified extension $\Fu$ of $F$. Let $A$ be a maximal $F$-split torus in $G$, $S$ maximal $\Fu$-split $F$-torus that contains $A$. Let $T= Z_{G}(S)$. Note that $X_*(S) = X_*(T)^{\Gal(\Omega/\Fu)}$ and $X_*(A) = X_*(T)^{\Gal(\Omega/F)}$. 
\begin{lemma} The simplicial isomorphism \[\cA_m: \cA(S, \Fsh) \rightarrow \cA(S',\FFsh)\] of Lemma \ref{ACLF} is $\Del_m$-equivariant. 
\end{lemma}  
\begin{proof} This is clear from the proof of Proposition \ref{ACLF}, Section \ref{CSCLF}, and Lemma \ref{charQSCLF}. 
\end{proof}
Let $\sigma \in \Gal(\Fsh/F)$ be as in Section \ref{etaledescent}. Let $\cF$ be a facet in $X_*(A)$. Then $\cF$ corresponds to a $\sigma$-stable facet $\tilde{\cF}$ in $X_*(S)$. Note that $\Del_m$ induces isomorphisms
\[\Gal(\Fsh/F) \cong \Gal(F_s/F)/I_F \cong \Gal(F_s'/F')/I_{F'} \cong \Gal(\FFsh/F').\] 
Let $\sigma' = \Del_m(\sigma)$ under this isomorphism.   Let $\tilde{\cF}' = \cA_m(\tilde\cF)$ and $\cF' =\tilde{\cF'}^{\sigma'}$.

\begin{proposition}\label{QSDP} The isomorphism 
\[\tp_m: \cP_{\tilde\cF} \times_{\fO_\Fsh} \fO_\Fsh/\fp_\Fsh^m \rightarrow \cP_{\tilde\cF'} \times_{\fO_{\FFsh}} \fO_{\FFsh}/\fp_{\FFsh}^m\]
has the property that $\sigma' \circ \tp_m = \tp_m \circ \sigma$. 
\end{proposition}
\begin{proof} Recall that the cocycle $s_{G}$ has been chosen to take values in $\Aut(H)$ and $s_{G}\rightarrow s_{G'}$ via  Lemma \ref{QSF}. Further, $\cT$ is defined over $\fO_F$ and $\cT_{\fO_\Fsh} = \cT \times_{\fO_F} \fO_\Fsh$. From this it is clear that $\sigma' \circ \tp_m = \tp_m \circ \sigma$ on $\cT \times_{\fO_\Fsh} \fO_\Fsh/\fp_\Fsh^m$.   In addition, using the fact that Chevalley-Steinberg systems on $G$ and $G'$ have been chosen compatibly (see Section \ref{CSCLF}),  it is easy to see that $\sigma' \circ \tp_m = \tp_m \circ \sigma$ on $\cU_{\tilde\cF} \times_{\fO_\Fsh} \fO_\Fsh/\fp_\Fsh^m$. This completes the proof of the proposition. 
\end{proof}

\section{Inner forms of quasi-split groups over close local fields} Let $F$ be a non-archimedean local field and let $G$ be a connected reductive group over $F$. Then there is a quasi-split group $G_{q}$ defined over $F$ such that $G$ is an inner form of $G_{q}$. In particular, the $F$-isomorphism class of $G$ is determined by an element in $H^1(\Gamma_F, G_{q}^{ad}(F_s))$. Moreover if $[G] \in E(F, G_0)_m$ then $[G_{q}] \in E(F, G_0)_m$ and $[G]$ is determined by an element of $H^1(\Aut(\Omega/F), G_{q}^{ad}(\Omega))$ (Recall that $\Omega = (F_s)^{I_F^m}$). Let $s_{G_q}$ be the element of $H^1(\Gamma_F/I_F^m, \Aut(R, \Delta)^{I_F^m})$ that determines $(G_q, B_q, T_q)$, up to $F$-isomorphisms. Let $G_{q}^{der}$ be the derived subgroup of $G_q$ and let $G_q^{ad}, G_q^{sc}$ denote the corresponding adjoint and simply connected groups. 
Then the groups $G_q^{der}$, $G_q^{ad}, G_q^{sc}$ are quasi-split (if $S_q$ is a maximal $F$-split torus in $G_q$ whose centralizer $T_q$ is a maximal torus, then $S_q \cap G_q^{der}$ is a maximal $F$-split torus of $G_q^{der}$ and $Z_{G_q^{der}}(S_q \cap G_q^{der}) = T_q \cap G_q^{der}$ is a maximal torus of $G_q^{der}$, similarly for $G_q^{ad}$ and $G_q^{sc}$) and are in fact forms of $G_0^{der}$, $G_0^{ad}$ and $G_0^{sc}$ respectively (to see this note that $G_q^{der} \times_F \Omega \cong (G_q\times_F \Omega)^{der}$ and $Z(G_q)\times_F \Omega \cong Z(G_q \times_F \Omega)$).  
 Using Proposition 13.1 (1) of \cite{Kot16} and the fact that $G_q^{ad}$ has trivial center, we have a canonical bijection \[ \kappa_{G_q}: H^1(\Aut(\Omega/F), G_{q}^{ad}(\Omega)) \rightarrow \left(X_*(T_q^{ad})/X_*(T_q^{sc})\right)_{\Aut(\Omega/F)}.\]
Let $E_i(F, G_q)_m$ denote the $F$-isomorphism classes of inner forms of $G_q$ that split over an at most $m$-ramified extension of $F$. Let $(G_q', B_q', T_q')$ correspond to the cocycle $q' \circ \fQ_m(s_{G_q})$ and let $E_i(F', G_q')_m$ be the corresponding object over $F'$. 
\begin{lemma}\label{IF} The congruence data $D_m$ induces an isomorphism
\[ \fI_m: \left(X_*(T_q^{ad})/X_*(T_q^{sc})\right)_{\Aut(\Omega/F)}\xrightarrow{\cong} \left(X_*({T'}_q^{ad})/X_*({T'}_q^{sc})\right)_{\Aut(\Omega'/F')}.\]
In particular, $D_m$ induces a bijection $E_i(F, G_q)_m \rightarrow E_i(F', G_q')_m$, $[G] \rightarrow [G']$ where $s_{G'} =\kappa_{G_q'}^{-1}\circ \fI_m\circ \kappa_{G_q}(s_{G})$. 
\end{lemma}
\begin{proof}
Note that $X_*(T_q) \cong X_*(T_0) \cong X_*(T_q')$ as $\bbZ$-modules  and the Galois action on $X_*(T_q)$ is determined by the cocycle $s_{G_q}$ (and similarly for $X_*(T_q^{ad}), X_*(T_q^{sc})$). Now the lemma is obvious by Lemma \ref{QSF}.
\end{proof}
To proceed, we need to prove a version of Lemma \ref{IF} at the level of cocycles. To do this, we will use some results from Section 2 of \cite{DR}.
\subsection*{Steinberg's vanishing theorem} Let $G$ be a connected, reductive $F$-group. Steinberg's vanishing theorem asserts that
\begin{theorem}[Theorem 56, \cite{Ste65}]\label{SVT} $H^1(\Gal(F_s/\Fu), G(F_s))=1.$
\end{theorem}
As a corollary of this theorem, we obtain that the natural surjection from $\Gal(F_s/F) \rightarrow \Gal(\Fu/F)$ induces an isomorphism
\[H^1(\Gal(\Fu/F), G(\Fu) )\cong H^1(\Gal(F_s/F), G(F_s)).\]

\subsection{Congruence data for inner forms; a comparison of cocycles}\label{IFCCLF} Let $A_q$ be a maximal $F$-split torus in $G_q$ and let $S_q$ be a maximal $\Fu$-split  $F$-torus in $G_q$ that contains $A_q$. Let $T_q = Z_{G_q}(S_q)$. Then $T_q$ is a maximal torus in $G_{q,\Fu}$ with maximal $\Fu$-split torus $S_q$. Let $C_q$ be an $\sigma$-stable alcove in $\cA(S_q,\Fu)$. 

Let $P_{C_q}$ be the Iwahori subgroup of $G_q^{ad}(\Fu)$ attached to $C_q$. Let $\Omega_{C_q}^{ad}\subset \tilde{W}^{ad}:= X_*(T_q^{ad})_{I_F} \rtimes W$ consist of elements which preserve the alcove $C_q$. Here $I_F$ is the inertia subgroup of $F$ and $W = W(G_{q,\Fu},S_{q,\Fu})$. Then
\begin{align}\label{omega}
\Omega_{C_q}^{ad} \cong \left(X_*(T_q^{ad})/X_*(T_q^{sc})\right)_{I_F}
\end{align}
by Lemma 15 of \cite{HR08}. Let $P_{C_q}^*$ be the normalizer in $G_q^{ad}$ of $P_{C_q}$.  Let $N_{C_q}^{ad}=N_{G_q^{ad}}(S_q^{ad})(\Fu) \cap P_{C_q}^*$. Then $\Omega_{C_q}^{ad}$ is the image of $N_{C_q}^{ad}$ in $\tilde{W}^{ad}$ and $\Omega_{C_q}^{ad} \cong P_{C_q}^*/P_{C_q}$. 

The following lemma is proved in Sections 2.3 and 2.4 of \cite{DR}. Although the authors assume that $G_{q,\Fu}$ is split in the beginning of Section 2.3 of \cite{DR}, this assumption is not needed in their proof of the following lemma. They use that when $G_{q,\Fu}$ is split, $\Omega_{C_q}^{ad} \cong X_*(T_q^{ad})/X_*(T_q^{sc})$ in Corollary 2.4.2 and Corollary 2.4.3; one should instead use \eqref{omega} when $G_{q,\Fu}$ is not necessarily split.

\begin{lemma}[Corollary 2.4.3, \cite{DR}] We have isomorphisms
\[H^1(\Gal(\Fu/F), \Omega_{C_q}^{ad}) \cong H^1(\Gal(\Fu/F), N_{C_q}^{ad} )\cong H^1(\Gal(\Fu/F), G_q^{ad}(\Fu)).\]
\end{lemma}

Let $c$ be a cocycle in  $Z^1(\Gal(\Fu/F), \Omega_{C_q}^{ad})$. By Lemma 2.1.2 of \cite{DR}, since $\Omega_{C_q}^{ad}$ is finite, we have
\[Z^1(\Gal(\Fu/F), \Omega_{C_q}^{ad}) = \Omega_{C_q}^{ad}.\]
Let $G$ be the inner form of $G_q$ determined by $c$. Let $c(\sigma)=w_\sigma$. Write $w_\sigma = (\lambda,w)$ with $\lambda \in X_*(T^{ad})_{I_F}$ and $w\in W$. Let $K \subset F_s$ denote the finite atmost $m$-ramified extension of $\Fu$ over which $G_{q,\Fu}$ splits. Let $t= Nm(\tilde\lambda(\pi_K))$ where $Nm: T_q^{ad}(K) \rightarrow T_q^{ad}(\Fu)$ and $\tilde \lambda \rightarrow \lambda$ under the usual surjection $X_*(T_q^{ad}) \rightarrow X_*(T_q^{ad})_I$. Let $\tw \in N_{G_q}(S_q)(\Fu)$ be the representative of $w$ chosen using the Chevalley-Steinberg system we fixed in Section \ref{CSCLF}.  

Let $m_\sigma =t\tw$. Since $w_\sigma$ stabilizes $C_q$, it follows that $m_\sigma P_{C_q} m_\sigma^{-1} = P_{C_q}$. Hence $m_\sigma \in P_{C_q}^*$. 
Therefore $\tc(\sigma)=m_\sigma \in Z^1(\Gal(\Fu/F), N_{C_q}^{ad}).$  
Denoting \[G(\Fu)\rightarrow G_q(\Fu), g_*\rightarrow g,\] the new  action of $\sigma$ on an element $g_*\in G(\Fu)$, which we denote by $\sigma_*$, is given by \[\sigma_* \cdot g_* = (\tc(\sigma)(\sigma \cdot g))_*\] (Here $\sigma \cdot g$ denotes the action of $\sigma$ on $g \in G_q(\Fu)$). Note that $c(\sigma) \in G_q^{ad}(\Fu) = Inn(G_q)(\Fu)$. The maximal $\Fu$-split torus $S_q$ of $G_q$ gives a maximal $\Fu$-split, $\Fu$-torus $S$ in $G$. Let $X_*(S) \rightarrow X_*(S_q), \tau_* \rightarrow \tau$.
For $\tau_* \in X_*(S)$, $\sigma_* \cdot \tau_* = (w_\sigma(\sigma \cdot \tau))_*$. Since $S_q$ is defined over $F$, $\sigma \cdot \tau \in X_*(S_q)$. Since $w_\sigma \in \Omega_{C_q}^{ad}$, we see that $X_*(S)$ is stable under the action of $\sigma$, and hence $S$ is defined over $F$.

\begin{lemma} \label{Fsplit}Let $A$ be the $F$-split torus of $G$ determined by the $\bbZ$-module $X_*(S)^{\sigma_*}$. Then $A$ is a maximal $F$-split torus in $G$.
\end{lemma}
\begin{proof}
Consider the reduced apartment $\cA(S_q,\Fsh)$. We view this as an apartment in the reduced building of $G(\Fsh)$ and denote it as $\cA(S, \Fsh)$. The action $\sigma_*$ on $x_* \in \cA(S, \Fsh)$ given by $\sigma_* \cdot x_* = (w_\sigma(\sigma \cdot x))_*$. Let $C_*$ denote the alcove in $\cA(S, \Fsh)$ corresponding to $C_q$. Then $\sigma_*\cdot C_* = (w_\sigma( \sigma \cdot C_q))_*$. Since $\sigma \cdot C_q = C_q$ and  since $w_\sigma \in \Omega_{C_q}^{ad}$, we see that $C_*$ is a $\sigma_*$-stable alcove in  $\cA(S,\Fsh)$. In particular, $\cA(S, \Fsh)$ is $\sigma_*$-stable. By Proposition 5.1.14 of \cite{BT2}, $C_*^{\sigma_*}$ is an alcove in the affine space $\cA(A, F)$. Since $\cA(A,F)$ contains a facet of maximal possible dimension, we see that $A$ is maximal $F$-split in $G$. 
\end{proof}

Let $(G_q', T_q', B_q', S_q')$  correspond to $(G_q,T_q,B_q,S_q)$ via congruence data $D_m$ as in Section \ref{QS}. By Lemma \ref{charQSCLF}, we have
\[\Omega_{C_q}^{ad} \cong \Omega_{C_q'}^{ad}.\]
Let $w_{\sigma'} \in \Omega_{C_q'}^{ad}$ be the image of $w_\sigma$ under this isomorphism. 
This isomorphism gives rise to a bijection of pointed sets
\begin{align}\label{matchco}
\fI_m: Z^1(\Gal(\Fu/F), \Omega_{C_q}^{ad}) &\rightarrow Z^1(\Gal(\FFu/F'), \Omega_{C_q'}^{ad}),\nonumber \\
c &\rightarrow c'
\end{align}
where $c'(\sigma') = w_{\sigma'}$. 
Let $m_{\sigma'} = t'\tw'$ where $w_{\sigma'} = (\lambda', w') \in X_*(T_q^{ad})_{I_{F'}} \rtimes W'$. Here $t'= Nm(\tilde\lambda'(\pi'_{K'}))$ where $Nm: T_q^{ad'}(K') \rightarrow T_q^{ad'}(\FFu)$ and $\tilde \lambda' \rightarrow \lambda'$ under the usual surjection $X_*(T_q^{ad'}) \rightarrow X_*(T_q^{ad'})_{I_{F'}}$, and $\tilde\lambda \rightarrow \tilde\lambda'$ under the isomorphism $X_*(T_q^{ad}) \cong X_*(T_q^{ad'})$. Also $\tw'$ is the representative of $w$ chosen using the Chevalley-Steinberg system fixed in Section \ref{CSCLF}. 
Let $\tc' \in Z^1(\Gal(\FFu/F'), N_{C_q'}^{ad'})$ be the cocycle with $\tc'(\sigma') = m_{\sigma'}$.

Let $G'$ be the inner form of $G_q'$ determined by $c'$ (or $\tc'$). Let $S'$ be the maximal $\FFu$-split, $\Fu$-torus of $G'$ corresponding to $S_q'$ but with the action of $\sigma'$ given by the cocycle $\tc'$. More precisely, for $g_*' \in G'(\FFu)$, 
\[\sigma'_* \cdot g'_* = (\tc'(\sigma') \cdot (\sigma' \cdot g'))_*\]
where $\sigma' =\Del_m(\sigma)$ as before, and $\sigma'\cdot g'$ denotes the action of $\sigma'$ on $G_q'(\FFu)$. 

  As in Lemma \ref{Fsplit}, we see that $S'$ is an $F'$-torus that is maximal $\FFu$-split and whose split component $A'$ is a maximal $F'$-split torus in $G'$. 
\begin{corollary} With $G \rightarrow G'$ as above, the $F$-rank of $G$ is equal to the $F'$-rank of $G'$. 
\end{corollary}
\begin{proof} This is because $\rank(S) = \rank(S')$ and the isomorphism $X_*(S) \rightarrow X_*(S')$ is $\sigma_*$-equivariant. Hence $\rank(A) = \rank(A')$ by Lemma \ref{Fsplit}. 
\end{proof}
\section{Congruences of parahoric group schemes; \'{e}tale descent}\label{PGSET}
The following lemma is easy.
 \begin{lemma}\label{IFACLF} The $\sigma$-equivariant isomorphism $\tilde\cA_{m}: \cA(S_q, \Fsh) \rightarrow \cA(S_q', \FFsh)$ induces a $\sigma_*$-equivariant isomorphism $\tilde\cA_{m,*}: \cA(S, \Fsh) \rightarrow \cA(S', \FFsh)$. 
\end{lemma}
Now let $\tilde\cF_*$ be $\sigma_*$-invariant facet in $\cA(S, \Fsh)$ and let $\tilde\cF'_{*}= \tilde\cA_{m,*}(\tilde\cF_*)$. Let $\cF_* =\tilde \cF_*^{\sigma_*}$ and $\cF'_{*} = \tilde\cF'^{\sigma'_{*}}$.

\begin{proposition}\label{EDPCLF} Let $m \geq 1$, $F, F'$ non-archimedean local fields with $e_F,e_{F'} \geq 2m$. Let $l$ as in Theorem \ref{QSP},  let $D_l$ be the congruence data of level $l$, and let $(G_q', T_q', B_q', S_q')$  correspond to $(G_q,T_q,B_q,S_q)$ via $D_l$. Let $\tp_m: \cP_{\tilde\cF} \times_{\fO_\Fsh} \fO_\Fsh/\fp_\Fsh^m \rightarrow \cP_{\tilde\cF'}\times_{\fO_{\FFsh}} \fO_{\FFsh}/\fp_{\FFsh}^m$ denote the $\sigma$-equivariant isomorphism of Theorem \ref{QSP} and Proposition \ref{QSDP}. Let $c \rightarrow c'$ via $\fI_m$ (see \eqref{matchco}).  The isomorphism $\tp_m$ induces a $\sigma_*$-equivariant isomorphism $\tp_{m,*}: \cP_{\tilde\cF_*}\times_{\fO_\Fsh} \fO_\Fsh/\fp_\Fsh^m \rightarrow \cP_{\tilde\cF'_{*}} \times_{\fO_{\FFsh}} \fO_{\FFsh}/\fp_{\FFsh}^m$. 
\end{proposition}
\begin{proof} 
 We begin by understanding the action of $\sigma_*$ on an element of $P_{\tilde\cF_*}$ more explicitly. Recall that
\[P_{\tilde\cF}=\left\langle \cU_{\tilde\cF}^+(\fO_{\Fsh}), \cT(\fO_{\Fsh}), \cU_{\tilde\cF}^-(\fO_{\Fsh}) \right\rangle\]
Let $g \in P_{\tilde\cF}$. Then $\sigma_* \cdot g_* = (m_\sigma (\sigma \cdot g) m_\sigma^{-1})_*$.
Let $b_0 \in \Phi^{red}(G_q,S_q)$ such that $2b_0$ is not a root. Let $y \in U_{b_0,\tilde\cF}$. Fix $\beta_0|_{S_q} = b_0$, fix the pinning $(L_{\beta_0}, x_{b_0})$ and write $y = x_{b_0}(u_0)$ for $u_0 \in L_{\beta_0}$ (As explained in Section \ref{splitting extension}, $L_{b_0} \cong L_{\beta_0}\hookrightarrow K$).
Let $\tilde\sigma$ denote a lift of $\sigma$ to $\Gamma_F$ and let $\beta = \tilde\sigma \cdot \beta_0$, $b=\sigma \cdot b_0$. 
Then we obtain a pinning $(L_\beta, x_b)$ from the pinning $(L_{\beta_0}, x_{b_0})$ via $\tilde\sigma$ and we have $\sigma \cdot x_{b_0}(u_0) = x_{b}(\tilde\sigma \cdot u_0)$; this follows using properties of Chevalley-Steinberg system recalled in Section \ref{CSsystem} (a), (b). Let $u= \tilde\sigma \cdot u_0$. Then $u \in L_\beta$. 
We need to compute $\tw x_b(u)\tw^{-1}$. We will first compute $\ts_a x_b(u) \ts_a^{-1}$ for $a\in \Delta$. Note that  
\begin{align}\label{tsa}
\ts_a=\prod_{\alpha \in \tilde\Delta_a} \ts_\alpha
\end{align} 
and that $L_{s_a \cdot b} = L_b$. Now for $\alpha_1,\beta_1 \in \Phi(G_q,T_q)$, we have $\ts_{\alpha_1} x_{\beta_1}(z)\ts_{\alpha_1}^{-1} = x_{s_{\alpha_1}(\beta_1)}(d_{\alpha_1, \beta_1}z)$ for all $z \in K$,  with $d_{\alpha_1, \beta_1 }= \pm 1$. Using the properties of Chevalley-Steinberg system recalled in Section \ref{CSsystem} (a), (b),  we have \begin{align}\label{galoisbeta}
d_{\alpha_1,\beta_1} = d_{\gamma(\alpha_1), \gamma(\beta_1)} \;\;\forall\;\; \gamma \in \Gal(K/\Fsh).
\end{align}
With $\beta$ as above, note that $\beta|_{S_q} = b$. Let \[d_{a,b}:=\prod_{\alpha \in \tilde\Delta_a}d_{\alpha,\beta}\] 
This notation is justified since \eqref{galoisbeta} implies that the definition of $d_{a,b}$ does not depend on the choice of $\beta$. Using the definition of $x_b$ in  \eqref{2anotroot}, a simple calculation yields that
$\ts_ax_b(u)\ts_a^{-1} = x_{s_a(b)}(d_{a,b}u).$ 
 Since we chose our Chevalley-Steinberg systems compatibly (see Section \ref{CSCLF}), we evidently have $d_{a,b} = d_{a',b'}$ for all $a \in \Delta, b \in \Phi(G_q,S_q)$. Iterating this process, we see that
$\tw x_b(u) \tw^{-1} = x_{w\cdot b}(d_{w,b}u)$ where $d_{w,b} = \pm 1$ and $d_{w,b}= d_{w',b'}$. 

Suppose $b_0 \in \Phi(G_q,S_q)$ such that $2b_0$ is a root.  Let $\beta_0, \bar\beta_0|_{S_q} = b_0$. Fix the pinning $(L_{\beta_0}, L_{\beta_0+\bar\beta_0}, x_{b_0})$ and write $y = x_{b_0}(u_0,v_0)$, with $u_0,v_0 \in L_{\beta_0}$ (Recall that $L_{b_0} \cong L_{\beta_0} \subset K$). Let $\beta = \tilde\sigma \cdot  \beta_0, \bar\beta = \tilde\sigma \cdot \bar\beta_0$ and $b = \sigma \cdot b_0$. We then obtain a pinning $(L_\beta, L_{\beta+\bar\beta}, x_b)$ via $\tilde\sigma$ and $\sigma \cdot x_{b_0}(u_0,v_0) = x_{\sigma \cdot b_0}(\tilde\sigma \cdot u_0, \tilde\sigma \cdot v_0)$ where $\tilde\sigma$ as before.
Let $u = \tilde\sigma \cdot u_0, v = \tilde\sigma \cdot v_0$. Then $u,v \in L_\beta$. 
We need to compute $\ts_a x_b(u,v) \ts_a^{-1}$ where $s_a$ is as in \eqref{tsa}.  
Let \[d_{a,b}:=\prod_{\alpha \in \tilde\Delta_a}d_{\alpha,\beta},\;\; , \;\;d_{a,2b}:=\prod_{\alpha \in \tilde\Delta_a}d_{\alpha,\beta+\bar\beta}\]
Again, the definitions of $d_{a,b}$ and $d_{a,2b}$ do not depend on the choice of $\beta$ by \eqref{galoisbeta}. 

Then a simple calculation yields
\begin{align*}
\ts_a x_b(u,v) \ts_a^{-1} = x_{s_a(b)}(d_{a,b}u, d_{a, 2b}v).
\end{align*}

Proceeding as in the previous case, we have 
$\tw x_b(u,v) \tw^{-1} = x_{w\cdot b}(d_{w,b}u,d_{w,2b}v))$ where $d_{w,b}, d_{w,2b} = \pm 1$ and $d_{w,b}= d_{w',b'}$ and $d_{w,2b} = d_{w',2b'}$.

Recall that $t = Nm(\tilde\lambda(\pi_K)) \in T_q^{ad}(\Fu)$. Then for each $\gamma\in \Gal(K/\Fu)$, $\gamma \cdot t = t$. Let $c \in \Phi^{red}(G_q,S_q)$ with $2c$ not a root. Let $\chi \in \Phi(G_q. T_q)$  with $\chi |_{S_q} =c$. Note that $\chi$ factors through $T_q^{ad}$. 
Fixing the pinning $(L_\chi, x_c)$ we have that $\chi:T \rightarrow G_m$ is defined over $L_\chi$ and $\chi(t) \in L_\chi^\times$. A simple calculation yields  $tx_c(u)t^{-1} = x_c(\chi(t)u)$ for each $u \in L_\chi$. If $c, 2c$ are roots, then with $\chi, \bar\chi$ such that $\chi |_{S_q}  =\bar\chi|_{S_q} =c$ and fixing the pinning $(L_\chi, L_{\chi+\bar\chi}, x_c)$, it follows that
$tx_c(u,v)t^{-1} = x_c(\chi(t)u, (\chi+\bar\chi)(t)v)$.
Hence, if $2b_0$ is not a root then \[\sigma_* \cdot (x_{b_0}(u_0))_* = (x_{w\cdot b}(d_{w,b}\chi(t)u))_*\] where $\chi|_{S_q} = w \cdot b.$ If $2b_0$ is a root, then 
\[ \sigma_* \cdot (x_{b_0}(u_0,v_0))_*= (x_{w\cdot b}(d_{w,b} \chi(t) u, d_{w,2b} (\chi+\bar\chi)(t)v))_*\]
where $\chi,\chi' \in \Phi(G_q,T_q)$ are such that $\chi, \bar\chi|_{S_q} = w\cdot b$. It is easy to calculate $\sigma_* \cdot (x_{2b_0}(0, v_0))_*$ using the observations above.  For $x \in \cT_q(\fO_{\Fsh}),$ 
\[\sigma_*\cdot x_* = (w (\sigma \cdot x)w^{-1})_*.\]
Combining these observations with the fact that $\tp_m$ is $\sigma$-equivariant (see Lemma \ref{QSDP}), it follows that the map $\tp_{m,*}$ has the property that
$\tp_{m,*} \circ \sigma_* = \sigma'_{*} \circ \tp_{m,*}$ (in this verification, we choose $\tilde\sigma'$ to correspond to $\tilde\sigma$ via $\Del_m$).
\end{proof}

\begin{corollary} The isomorphism $\tp_{m,*}$ induces an isomorphism of group schemes
\[ p_{m,*}: \cP_{\cF_*}\times_{\fO_F} \fO_F/\fp_F^m \rightarrow \cP_{\cF'_{*}} \times_{\fO_{F'}} \fO_{F'}/\fp_{F'}^m \times_{\psi_m^{-1}} \fO_F/\fp_F^m.\]
In particular
$\cP_{\cF_*}(\fO_F/\fp_F^m)$ and $\cP_{\cF'_{*}}(\fO_{F'}/\fp_{F'}^m)$ are isomorphic as groups.
\end{corollary}
\begin{proof} This follows from Proposition \ref{EDPCLF} and \'{e}tale descent (Example B, Section 6.2, \cite{BLR}).
\end{proof}

\end{document}